\documentclass[10pt,letterpaper]{article}
\usepackage{tcolorbox}

\usepackage{amssymb}
\usepackage{amsthm}
\usepackage{amsmath}
\usepackage{graphicx}
\usepackage{fullpage}
\usepackage[margin=1in]{geometry}
\usepackage{enumerate}
\usepackage{algorithm,algorithmic}
\usepackage{tikz}
\usepackage{subcaption}

\usepackage[colorlinks=red,linkcolor=red]{hyperref}
\makeatletter
\renewcommand*{\eqref}[1]{%
  \hyperref[{#1}]{\textup{\tagform@{\ref*{#1}}}}%
}
\makeatother

\usepackage{color}

\newtheorem{theorem}{Theorem}[section]
\newtheorem{lemma}[theorem]{Lemma}
\newtheorem{proposition}[theorem]{Proposition}

\newtheorem{corollary}[theorem]{Corollary}

\theoremstyle{definition}
\newtheorem*{definition}{Definition}

\theoremstyle{remark}
\newtheorem*{remark}{Remark}

\newcommand{\E}{\mathbb{E}}
\newcommand{\R}{\mathbb{R}}

\newcommand{\be}{\mathbf{e}}
\newcommand{\bu}{\mathbf{u}}
\newcommand{\bv}{\mathbf{v}}

\newcommand{\bx}{\mathbf{x}}
\newcommand{\by}{\mathbf{y}}
\newcommand{\bz}{\mathbf{z}}
\newcommand{\bzero}{\mathbf{0}}

\newcommand{\sd}{\mathbb{S}^{d-1}}

\DeclareMathOperator*{\argmin}{arg\,min}

\title{Optimal Regularization Under Uncertainty: \\ Distributional Robustness and Convexity Constraints} 
\author{Oscar Leong\thanks{Department of Statistics and Data Science, University of California, Los Angeles (email: \texttt{oleong@stat.ucla.edu})}
\and Eliza O'Reilly\thanks{Department of Applied Mathematics and Statistics, Johns Hopkins University (email: \texttt{eoreill2@jh.edu})} 
\and Yong Sheng Soh\thanks{Department of Mathematics, National University of Singapore (email: \texttt{matsys@nus.edu.sg})}
}
\begin{document}

\maketitle

\begin{abstract}
Regularization is a central tool for addressing ill-posedness in inverse problems and statistical estimation, with the choice of a suitable penalty often determining the reliability and interpretability of downstream solutions. While recent work has characterized optimal regularizers for well-specified data distributions, practical deployments are often complicated by distributional uncertainty and the need to enforce structural constraints such as convexity. In this paper, we introduce a framework for distributionally robust optimal regularization, which identifies regularizers that remain effective under perturbations of the data distribution. Our approach leverages convex duality to reformulate the underlying distributionally robust optimization problem, eliminating the inner maximization and yielding formulations that are amenable to numerical computation. We show how the resulting robust regularizers interpolate between memorization of the training distribution and uniform priors, providing insights into their behavior as robustness parameters vary. For example, we show how certain ambiguity sets, such as those based on the Wasserstein-1 distance, naturally induce regularity in the optimal regularizer by promoting regularizers with smaller Lipschitz constants. We further investigate the setting where regularizers are required to be convex, formulating a convex program for their computation and illustrating their stability with respect to distributional shifts. Taken together, our results provide both theoretical and computational foundations for designing regularizers that are reliable under model uncertainty and structurally constrained for robust deployment.

\end{abstract}

\section{Introduction}

Across many real-world tasks in data science and machine learning, it is necessary to quantify and understand the potential uncertainty in a given model. Such uncertainty could be due to a number of factors, such as limited observations, dynamic environments, or modeling errors. These considerations are especially prevalent in problems in which solution reliability and robustness are of critical importance due to safety concerns, such as medical imaging. Given such considerations, we may wish to enforce that a given model we learn is provably robust to certain perturbations or exhibits beneficial properties via structural constraints that may aid in solution reliability. Techniques for ensuring robustness in problems in data science has a rich history, with powerful techniques from areas such as robust statistics \cite{loh2024theoretical} and Distributionally Robust Optimization (DRO) \cite{kuhn2019wasserstein}.

 In this work, we are particularly interested in questions of robustness and uncertainty in the context of statistical estimation and inverse problems, where the goal is to recover an underlying data signal from corrupted observations. To address the ill-posedness present in these problems, it is common to augment a data fidelity term with a regularization penalty to promote certain structure in a solution. The choice of regularizer is critical, as it governs both reconstruction accuracy and computational tractability.

The literature is rich with a variety of possible regularizers to choose from. Classical examples include hand-crafted regularizers that promote structures such as sparsity \cite{Tao2006, Daubechiesetal04, Donoho2006, Tibshirani94}, low-rankness \cite{CandesRecht09, FazelThesis, Rechtetal10}, or smoothness \cite{Rudinetal92}. The performance of such regularizers in specific inverse problems has been studied extensively, with many results focusing on estimation guarantees with respect to properties such as sample complexity and robustness to noise \cite{chandrasekaran2012convex, oymak2016sharp}. However, these guaranteees are mainly in settings where the regularizer is perfectly tailored to the underlying signal's structure (e.g., a sparsity-inducing norm to reconstruct sparse vectors). When the underlying structure is more difficult to characterize, data-driven regularizers are preferred as they can be tailored for a data distribution of interest. Such regularizers, however, often lack theoretical guarantees in the context of inverse problems as the particular model structure they learn is not well-understood. More recently, some works have aimed to understand the learned structure of such regularizers and consider whether a given regularizer is ``optimal'' for a given data distribution \cite{leong2024star, leong2025optimal, traonmilin2024towards, traonmilin2024theory,    zhang2025learning}.

While these results offer insights into the structure of such regularizers, these works operate in a well-specified setting, where the underlying data distribution or signal structure is known exactly. For instance, if one seeks to guard against distribution shifts at inversion time, it is unclear how one should design such a regularizer. Moreover, to further increase robustness and solution reliability, it may be useful to enforce structural constraints, such as convexity, in designing such robust regularizers. Given concerns regarding uncertainty and solution reliability, we aim to understand how to meaningfully integrate distributional robustness and structural constraints in the design of regularizers. In particular, we study the following questions:  
\begin{center}
    \textit{How do we compute an ``optimal'' regularizer when 1) the underlying data distribution is itself uncertain, and 2) we wish to enforce modeling constraints (e.g., convexity) for reliable downstream solutions?}
\end{center}

\subsection{Uncertainty Modeling via Distributionally Robust Optimization}

We address these questions rigorously through the framework of DRO. To describe our setting, let $\mathcal{F}$ denote a family of regularization functionals and define a criterion $\mathcal{L}(f;P)$ that measures how effectively $f$ captures the structure of a data distribution $P$; smaller values of $\mathcal{L}(f;P)$ correspond to better regularizers. To ensure robustness, we require that the regularizer remain optimal in a worst-case sense, performing well across all admissible perturbations of $P$. Concretely, given a divergence $d(\cdot,\cdot)$ between probability measures and a tolerance $\epsilon \geq 0$, we study a problem of the form \begin{align*}
    \argmin_{f \in \mathcal{F}} \left[\max_{d(Q,P) \leq \epsilon}\mathcal{L}(f;Q) \right].
\end{align*}

Intuitively, this formulation seeks a regularizer that promotes structure not only for the nominal distribution $P$, but also for all nearby distributions. To make progress on understanding solutions to this problem, we fix a family of regularizers and specify an appropriate criterion. Following \cite{leong2025optimal}, we consider regularization functionals $f$ that are continuous, positive except at the origin, and positively homogeneous (i.e., $f(t \bx) = tf(\bx)$ for all $t \geq 0$). This family of regularizers is expressive, as it includes all norms along with nonconvex quasinorms, such as the $\ell_q$-quasinorm for $q \in (0,1)$. Moreover, this set of conditions specifies regularizers $f$ as the gauge function (or Minkowski functional) of a \textit{star body} $K \subset \R^d$:
\begin{equation} \label{eq:gauge}
f(\bx) = \| \bx \|_{K} := \inf \{ \lambda \geq 0 : \bx \in \lambda \cdot K \}.
\end{equation} A set $K$ is called a star body if it is compact, has non-empty interior, and for any $\bx \neq 0$, the ray $\{\lambda \bx : \lambda \geq 0\}$ intersects the boundary of $K$ exactly once.

For our criterion, we propose to analyze $$\mathcal{L}(f;P) := \mathbb{E}_{P}[f(\bx)].$$ The above objective provides a meaningful criterion for regularizer selection as an effective choice of regularization function $f$ is one that evaluates to small values whenever the input is {structured}; that is, it resembles data of interest.  Conversely, it should evaluate to large values on inputs that are {unstructured}.  The objective $\mathbb{E}_{P}[ f(\bx) ]$ captures this criteria -- it seeks functions $f$ that evaluates to small values on input equal to data drawn from $P$, and penalizes for inputs that appear different from data drawn from $P$. Many data-driven regularization frameworks have used similar objectives to learn regularizers from data, such as those based on dictionary learning \cite{gribonval2015sample} and adversarial regularization \cite{lunz2018adversarial}.

In summary, the central optimization problem of interest in this work is the following: \begin{equation} \label{eq:dro}
\underset{K\ \text{star body}}{\mathrm{argmin}} ~\Big[ \underset{ d (Q,P) \leq \epsilon}{\max} ~ \E_Q[\|\bx \|_K] ~ \Big] \qquad \mathrm{s.t.} \qquad \mathrm{vol}(K) = 1. 
\end{equation} We finally remark that we include an additional normalization constraint $\mathrm{vol}(K)=1$. Normalization is necessary, as without it the optimal solution would be trivial (the zero function). Additionally, normalization encourages solutions $f$ that evaluate to large values over inputs that are unstructured. We will also show in this work that this normalization leads to reasonable solutions and the resulting formalization \eqref{eq:dro} is frequently expressible as a convex program.

\subsection{Our Contributions}

In Section \ref{sec:motivation-robustness}, we will discuss several issues that support the need for robustness considerations in learning regularizers. Then in Section \ref{sec:dro-reformulation-convex-duality}, we show that the DRO formulation \eqref{eq:dro} exhibits an equivalent optimization formulation that eliminates the inner maximization in \eqref{eq:dro}. Previous work analyzed a simpler optimization problem with $\epsilon=0$ and showed that one can use dual Brunn-Minkowski theory to characterize minimizers of the objective. We show in this work that a more direct analysis of the optimization problem using convex duality can lead to a simpler problem whose solution can be numerically computed. We will introduce the intuition behind this in Section \ref{sec:dro-reformulation-convex-duality}, illustrate the optimal solutions via numerical examples in Section \ref{sec:numerics-dro}, along with how the choice of divergence $d(\cdot,\cdot)$ and tolerance $\epsilon$ plays a role in the DRO solution in Section \ref{sec:structural-properties}. Notably, our results hold for any input distribution, including empirical measures and distributions with low-dimensional supports, which is in stark contrast to prior work \cite{leong2025optimal, leong2024star}. Then, we show in Section \ref{sec:convex} how we can use these ideas to analyze the optimal regularizer for a distribution under the additional constraint that the regularizer is assumed to be convex. We give a description of a convex program to compute the level sets of such regularizers and discuss several examples. Finally, in Section \ref{sec:extensions} we discuss how our proof techniques can be used to give elementary arguments for prior results on optimal regularization and we will highlight extensions of our theory to variants of the criterion functional $\mathcal{L}(f;P)$, such as those learned in adversarial regularization.

\subsection{Related Work}

\paragraph{Robustness of regularizers.} The robustness literature for regularizers in inverse problems largely examines the sensitivity of specific estimators to noise and tuning. For $\ell_1$-type methods, a series of works quantify sharp phase transitions and risk bounds under different regularization strengths, as well as oracle-type stability bounds in noisy regimes (e.g., \cite{berk2021best, negahban2012unified, oymak2016sharp, wojtaszczyk2010stability}). A complementary thread introduces different data-fit terms to adapt to different types of noise or unknown noise levels in inverse problems, such as absolute deviation estimators \cite{bloomfield1980least} or the square-root LASSO \cite{belloni2011square}. We additionally note recent work \cite{olea2022out} that connects square-root LASSO with a convex penalty to distributionally robust optimization, and gives guarantees on the out-of-sample performance of such estimators along with prescriptions on the choice of regularization strength.

Beyond noise, robustness under model/regularizer misspecification has been analyzed in compressed sensing with basis or grid mismatch \cite{pezeshki2015compressed, tang2013compressed}. Recent works \cite{shoushtari2024prior} show that plug-and-play denoisers as regularizers can perform well despite small distribution shifts and can exhibit performance gains from modest in-domain adaptation \cite{shoushtari2024prior}. Robustness to distributional shifts have also been investigated for generative modeling-based priors, such as those given by normalizing flows \cite{asim2020invertible}. By contrast, \textit{explicitly designing regularizers to mitigate against distributional shifts} remains nascent; our formulation addresses this gap via a distributionally robust objective and convex-duality reformulations.

\paragraph{Optimal regularization.} Recent work asks which regularizer is optimal for a given dataset or inverse problem. For quadratic/Tikhonov families, closed-form optimal functionals and learning schemes are available \cite{alberti2021learning,chirinos2024learning}, and there is a parallel literature on bilevel parameter learning for variational imaging \cite{ calatroni2017bilevel,ehrhardt2023optimal,kunisch2013bilevel}. Beyond parameter choice, recent work \cite{traonmilin2024towards,traonmilin2024theory} characterizes, for a model set and linear measurement operator class, which convex penalty is optimal. This theory recovers canonical instances such as the $\ell_1$-norm for sparsity. Closer to our setting, the works \cite{ leong2024star,leong2025optimal} show that among continuous, positively homogeneous functionals, the optimal gauge for a given data distribution admits geometric characterizations using star geometry and dual Brunn-Minkowski theory. Our work extends this line by incorporating distributional robustness and convexity constraints, yielding computable programs whose solutions interpolate between data-adapted and uniform priors.



\section{Preliminaries} \label{sec:prelims}

We briefly introduce certain geometric concepts that are used in this paper.  For a deeper treatment of this topic, we refer the interested reader to \cite{hansen2020starshaped} for a survey on star geometry, and \cite{schneider2013convex} for a reference to convex geometry.

Given $\bx,\by \in K$, we let $[\bx,\by]$ denote the line segment connecting $\bx$ and $\by$.  We say that a set $K \subset \mathbb{R}^{d}$ is {\em convex} if $[\bx,\by] \subset K$ for all $\bx,\by \in K$.  We say that $K$ is {\em star} if $[0,\by] \subset K$ for all $\by \in K$.  We call a compact star $K$ a {\em star body} if has nonempty interior and for every $\bx \neq 0$, the ray $\{\lambda \bx : \lambda > 0\}$ intersects the boundary of $K$ exactly once. The set of all star bodies in $\R^d$ is denoted by $\mathcal{S}^d.$ A set $K$ is called a \textit{convex body} if it is compact, convex with non-empty interior such that $0 \in \mathrm{int}(K)$. We say that a point $\bx$ {\em sees} $\by$ if $[\bx,\by] \in K$.  The {\em star} of $\bx$ are all points that $\bx$ sees; i.e., $\mathrm{st}(\bx:K) = \{ \by \in K : [\bx,\by] \in K \}$.  In particular, $K$ is star if $\mathrm{st}(0:K) = K$.  The {\em kernel} of $K$ are points that see all of $K$; that is, $\mathrm{ker}(K) = \{ \bx : \mathrm{st}(\bx:K) = K \}$.  A star set $K$ is convex if and only if $\mathrm{ker}(K) = K$.

Let $\sd$ and $B^d$ denote the unit Euclidean sphere and ball in $\R^d$, respectively. Suppose $K$ is a star body. Its {\em radial function} $\rho_K : \sd \rightarrow \mathbb{R}$ is defined by
\begin{equation*}
\rho_K (\bu) : = \sup \, \{ \, \lambda \geq 0 : \lambda \cdot \bu \in K \, \}.
\end{equation*}
A consequence is that if $\rho_K$ is continuous and positive over $\sd$, then $K$ is a compact star body  \cite{HHMM:20}. Note that for any two star bodies $K,L$, we have that $K \subseteq L$ if and only if $\rho_K \leq \rho_L$. The reciprocal of the radial function is called the {\em gauge function}
\begin{equation*}
\| \bx \|_K := \inf \{ \lambda \geq 0 : \bx \in \lambda \cdot K \}.
\end{equation*} We let $\|f\|_{\infty} := \sup_{\bx \in \sd} |f(\bx)|$ denote the supremum norm of $f$ over the sphere $\sd$. We will also consider dual mixed volumes between star bodies in this work. In particular, for $i \in \R$ and star bodies $K,L \in \mathcal{S}^d$, we define $\Tilde{V}_i(K,L)$ as the \textit{$i$-th dual mixed volume} between $K$ and $L$: $$\Tilde{V}_i(K,L) = \frac{1}{d}\int_{\sd} \rho_K(\bu)^i \rho_L(\bu)^{d-i}\mathrm{d}\bu.$$ Note that we recover useful identities in certain cases, such as $\tilde{V}_i(K,K)=\mathrm{vol}(K)$ for any $i \in \R$, where $\mathrm{vol}(\cdot)$ is the usual $d$-dimensional volume. In the special case, $i=-1$, the following result gives a concrete lower bound on the dual mixed volume, along with a characterization of the equality cases: \begin{theorem}[Special Case of Theorem 2 in \cite{lutwak1975dual}] \label{thm:lutwak-dmv} For star bodies $K,L \in \mathcal{S}^d$, we have
\[\tilde{V}_{-1}(K, L)^d \geq \mathrm{vol}(K)^{-1}\mathrm{vol}(L)^{d + 1},\]
and equality holds if and only if $K$ and $L$ are dilates, i.e., there exists an $\alpha > 0$ such that $K = \alpha L$.
\end{theorem}

\section{Distributionally Robust Optimal Regularizers} \label{sec:dro}


In this section, we study the DRO formulation \eqref{eq:dro} of the optimal regularization problem. Our main result is an alternative, but equivalent formulation of \eqref{eq:dro}; in particular, it is one that is amenable to computation. Using these formulations, we study how the distributionally robust optimal regularizers behave and exhibit robustness to changes in the underlying distribution.

\subsection{Motivation for Robustness} \label{sec:motivation-robustness}

We first discuss potential issues that may arise if we do not take robustness considerations into account. To begin, let us recall the initial criterion of finding an optimal regularizer over the space of star bodies. The following theorem provides a concrete characterization of the optimal star regularizer for certain well-behaved distributions, which depends on a particular functional that captures the mass of the distribution in any given direction and defines a new, data-dependent star body.

\begin{theorem}[Theorem 3 in \cite{leong2025optimal}] \label{thm:optstarbodyreg}
Let $P$ be a distribution on $\R^d$ with density $p$ and $\mathbb{E}_P[\|\bx\|_2] <\infty$. Consider the following optimization problem: \begin{align}
   \min_{K\ \text{star body}}\ \mathbb{E}_P[\|\bx\|_K]\ \qquad \mathrm{s.t.} \qquad \mathrm{vol}(K) = 1  \label{eq:opt_star_formulation}
\end{align}  Define the function $\rho_P$ over the unit sphere $\mathbb{S}^{d-1}$:
\begin{equation} \label{eq:rhoP}
\rho_{P}(\mathbf{u}) := \left(\int_0^{\infty} r^d p(r\mathbf{u}) \mathrm{d} r\right)^{1/(d+1)}, \qquad \mathbf{u} \in \mathbb{S}^{d-1}. 
\end{equation}
Suppose $\rho_P$ is positive and continuous.  Let $L_P$ be the star body whose radial function is $\rho_P$.  Then $\hat{K}$, as defined below, 
is the unique minimizer to \eqref{eq:opt_star_formulation}:
\begin{equation} \label{eq:optstar}
\hat{K} := \mathrm{vol}(L_P)^{-1/d}L_P.
\end{equation}
\end{theorem}

This result first appears in \cite{leong2025optimal}, where the proof appeals to dual mixed volumes and dual Brunn-Minkowski theory \cite{lutwak1975dual}.  In particular, the authors show that the objective \eqref{eq:opt_star_formulation} can be interpreted as a (dual) mixed volume, and by exploiting dual mixed volume inequalities such as Theorem \ref{thm:lutwak-dmv} and reading off equality conditions, one obtains descriptions of the optimal regularizer.

While the result provides strong insights into the form of the optimal regularizer for certain distributions, we highlight pathologies that arise in the absence of robustness considerations in the original formulation.

{\bf Atomic measures and memorization.}  
We note that the optimal star body regularizer has the interpretation of \textit{memorizing} data. This is particularly clear for data distributions $P$ given by atomic measures, which do not satisfy the assumptions of Theorem \ref{thm:optstarbodyreg}. A significant reason for this is that there is no minimizer for the above problem in this case. To see this, consider a data distribution that is uniformly supported on the standard basis vectors $\{ \pm \be_i \}_{i=1}^{d}$. We argue that the optimal objective value is zero. Construct the following cylinder in $\mathbb{R}^{d}$ with unit-volume for some parameter $\sigma > 0$: 
\begin{equation*}
T_{1,\sigma} := \left\{ \bx = (x_1,\dots,x_d) : |x_1| \leq 1/(2\sigma), \| (x_2,\ldots,x_{d-1})^T \|_2 \leq c \sigma^{1/(d-1)} \right\},
\end{equation*}
with $c$ chosen so that $T_{1,\sigma}$ has volume $1/d$.  Define the analogous sets $T_{i,\sigma}$ for $i \in [d]$ and put
$$T_{\sigma} := \bigcup_{i=1}^{d} T_{i,\sigma},$$
which has volume approximately one.  Cylinders are star bodies, and hence so is $T_{\sigma} $. Take $\sigma \rightarrow 0$.  Because the volume of the overlap between these cylinders vanish, $\mathrm{vol}(T_{\sigma}) \rightarrow 1$.  One then has
\begin{equation*}
\| \be_i \|_{T_{\sigma}} = 2 \sigma,
\end{equation*}
and hence the objective $\mathbb{E} [\| \bx \|_{T_{\sigma}}] \rightarrow 0$ as $\sigma \rightarrow 0$.  As such, the optimal objective value is zero.  But a zero objective cannot be attained by a star body, for if so, it must be that $\| \be_i \|_{T_{\sigma}} = 0$, and hence $\rho_{T_{\sigma}} ( \be_i ) \to +\infty$, which forces $T_{\sigma}$ to be unbounded.

While this example concerns data supported on standard basis vectors, the same argument extends to any atomic measure: if $P$ is an empirical distribution, then an optimal star body regularizer cannot exist. The reason it cannot exist is that the star set associated to such a distribution has zero (Lebesgue) volume, which does not allow it to satisfy the normalization constraint in \eqref{eq:opt_star_formulation}.

More generally, the fact that the optimal regularizer $\|\cdot\|_{\hat{K}}$ memorizes data can also be seen through the definition of the \textit{summary statistic} defined in \eqref{eq:rhoP}: the function $\rho_P$ summarizes data along radial directions in the sense that $\rho_P(\bu)$ quantifies the density of the data distribution $P$ that lies along a single direction $\bu$, along with how far this mass lies from the origin. This is perhaps useful in the case when $P$ has a well-behaved density, but less so in the case previously discussed where $P$ is an empirical measure.

On the surface, the fact that \eqref{eq:rhoP} memorizes is undesirable, because the regularizer $\hat{K}$ does not appear to learn the low-complexity structure that may be present in $P$. However, this may also be expected, since we have given $\hat{K}$ the flexibility to be any nonconvex gauge regularizer, which is an extremely expressive family of models; in particular, it necessarily means that $\hat{K}$ has been provided with the ability to overfit.

{\bf Ill-posedness.}  A closely related point we make is that the gauge function evaluations corresponding to the optimal star regularizer are sensitive to small changes in $P$. Let $P$ be the uniform distribution over the set of standard basis vectors $\mathcal{E}:=\{\be_1,\dots,\be_d\}$.  Let $\mathcal{E}'$ be a different set of vectors obtained by slightly perturbing the standard basis vectors; for concreteness, for small $\epsilon_i > 0$, consider $\be_1' := (1+\epsilon_1, \epsilon_2, \epsilon_3, \ldots, \epsilon_d)$.  

Let $P'$ be the uniform distribution over $\mathcal{E}'$.  Because these are atomic measures, the optimal star regularizer does not exist for $P$ and $P'$.  Let $P_\sigma$ be the distribution obtained by convolving $P$ with the Gaussian kernel with bandwith $\sigma$:
\begin{equation}
P_{\sigma} = P \ast \mathcal{N}(0, \sigma^2I),
\end{equation}
and define $P_\sigma'$ similarly. In this case, optimal star regularizers exists for $P_{\sigma}$ and $P_\sigma'$.  However, consider the gauge function evaluation over, say, $\be_1$. We would then have $\| \be_1 \|_{P_{\sigma}} \rightarrow 0$ as $\sigma \rightarrow 0$ -- the optimal star regularizer with respect to $P$ (in a sense) is the indicator function on the standard basis vectors.  However, by the same reasoning, the optimal star regularizer with respect to $P'$ is the indicator function on $\mathcal{E}'$, and hence $\| \be_1 \|_{P'_{\sigma}} \rightarrow \infty$ as $\sigma \rightarrow 0$. This is despite the fact that the data points $\mathcal{E}$ and $\mathcal{E}'$ are close to one another. Hence this example illustrates that such optimal star regularizers can be sensitive to the input distribution: nearby distributions can lead to drastically different optimal gauge functions. 

\subsection{DRO Reformulation via Convex Duality}
\label{sec:dro-reformulation-convex-duality}
Given these considerations, we would like to argue that exploiting a distributionally robust formulation will be beneficial in the sense that (i) one can show existence of solutions for any distribution, but also (ii) robustness considerations equip the optimal regularizer with additional regularity benefits both empirically and theoretically. To show this, we consider the DRO problem \eqref{eq:dro} with the ambiguity set defined via the Wasserstein distance. Given a cost function $C(\bx, \by)$ and distributions $P,Q$, we define $$d_W(Q,P) := \inf_{\beta \in \Gamma(Q,P)} \mathbb{E}_{(X,Y) \sim \beta} \left[C(X,Y)\right]$$ where $\Gamma(Q,P)$ is the set of all couplings between $Q$ and $P$.  Here, $C$ models a reasonable choice of cost function -- minimally, it should satisfy (i) $C(\bx,\bx) = 0$ for all $\bx$, (ii) $C(\bx,\by) > 0$ for all $\bx \neq \by$, and (iii) lower semi-continuity. Common examples include powers of $\ell_q$-norms, $C(\bx, \by) := \|\bx - \by\|^{\alpha}_q$ for $q,\alpha \geq 1$. For a given cost $C$, will consider the following problem for the remainder of this section: \begin{equation} \label{eq:dro-wass}
\underset{K\ \text{star body}}{\mathrm{argmin}} ~\Big[ \underset{ d_W (Q,P) \leq \epsilon}{\max} ~ \E_Q[\|\bx \|_K] ~ \Big] \qquad \mathrm{s.t.} \qquad \mathrm{vol}(K) = 1. 
\end{equation} 

Obtaining exact characterizations of the optimal star body solving \eqref{eq:dro-wass} is challenging, as the optimal distribution solving the inner maximization problem will depend on the optimization variable $K$ in a highly non-trivial fashion for most cases of $P,\epsilon,$ and $d_W$. We will investigate specific examples where we can make more concrete claims about the optimal solution in Section \ref{sec:structural-properties}. Instead, what we show is that there exists a reformulation of the above optimization problem using convex duality that is amenable to numerics, allowing us to visualize the optimal distributionally robust regularizer in several settings. Our main result is as follows:

\begin{theorem} \label{thm:dro_formulation} 
Let $P$ be a distribution on $\R^d$ with $\mathbb{E}_P[\|\bx\|_2]<\infty$ and suppose $C : \R^d \times \R^d \rightarrow \R$ is a non-negative, lower semi-continuous cost function satisfying $C(\bx,\by) = 0$ if and only if $\bx=\by$. Then
the optimization formulation \eqref{eq:dro-wass} is equivalent to the following
\begin{equation} \label{eq:dro_cts}
\underset{K, s, \lambda \in L1 (\mathrm{d}P)}{\mathrm{argmin}} ~~ s \epsilon + \int \lambda(\bx) \mathrm{d} P (\bx) \quad \mathrm{s.t.} \quad s C(\bx,\by) + \lambda (\bx) \geq \| \by \|_{K} , s \geq 0, \mathrm{vol}(K) \leq 1
\end{equation} where $L1(\mathrm{d}P) := \{f : \int |f(\bx)| \mathrm{d}P(\bx) < \infty\}$
\end{theorem} 
In addition to $K$ being an optimization variable (as in \eqref{eq:dro-wass}), it is necessary to introduce the additional variables $s$, which is a scalar variable, and $\lambda$, which is a function in $\bx$.  The main utility that Theorem \eqref{thm:dro_formulation} offers over \eqref{eq:dro-wass} is that it eliminates the inner maximization within \eqref{eq:dro-wass}.  In particular, we are able to compute (approximately) optimal solutions to \eqref{eq:dro-wass} by suitably discretizing \eqref{eq:dro_cts}.  We illustrate this process with numerical experiments in Section \ref{sec:numerics-dro}.

We explain how one arrives at the formulation in \eqref{eq:dro_cts}.  In essence, the key idea is to apply convex duality to the inner maximization in \eqref{eq:dro-wass}.  To provide some intuition, let $\mathcal{U} = \{ U : U \subset \sd \}$ be a collection of open subsets of the unit sphere $\sd$ that form a partition of the sphere $\sd$, up to a set of zero measure.  In what follows, we seek the optimal star regularizer among the collection of star sets $K$ whose radial functions are piecewise {\em constant} over each $U \in \mathcal{U}$.  With a slight abuse of notation, we simply say such sets $K$ are piecewise constant over $\mathcal{U}$.  We let $\{ t_U : U \in \mathcal{U} \}$ denote the gauge function of $K$ in the direction $U$.  In particular, these $t_U$'s will operate as our main decision variables.  In addition, we assume that the sets in the partition have equal area so that the volume of $K$ scales with $\sum_{U \in \mathcal{U}} t_U^{-d}$.  We restrict $P$ and $Q$ to be atomic measures that take on precisely one value within each $U$.  More concretely, suppose we let $\mathcal{V}$ denote the collection of all possible realizations of the support of $P$ and $Q$
$$
\mathcal{V} := \{ \bv_{U} : U \in \mathcal{U} \} \subset \mathbb{R}^{d}.
$$
The collection $\mathcal{V}$ satisfies $\bv_{U} \in U$ for all indices $U \in \mathcal{U}$.  With these assumptions in place, the finite-dimensional analog of \eqref{eq:dro-wass} can be written as
\begin{equation} \label{eq:dro_discretized}
\underset{t_U}{\mathrm{argmin}} ~\Big[ \underset{ d_{W} (\boldsymbol{p}, \boldsymbol{q}) \leq \epsilon}{\max} ~ \sum_{U \in \mathcal{U} } \mathbb{P} [Q = \bv_U ] \| \bv_U \|_2 t_U ~ \Big] \qquad \mathrm{s.t.} \qquad \sum t_U^{-d} \leq 1. 
\end{equation}
In particular, because $P$ and $Q$ are atomic distributions, we can express these as finite dimensional vectors.  In the above, we denote $P$ and $Q$ as $\boldsymbol{p},\boldsymbol{q} \in \mathbb{R}^{|\mathcal{U}|}$.  In addition, we obtain the expression for the objective in \eqref{eq:dro_discretized} by noting the following
$$
\mathbb{E}_{\bx \sim Q} [ \| \bx \|_{K} ] = \sum_{U \in \mathcal{U}} \mathbb{P}[Q = \bv_U] \| \bv_U \|_{K} = \sum_{U \in \mathcal{U}} \mathbb{P}[Q = \bv_U] \| \bv_U \|_{2} t_U.
$$

Now suppose that the variables $t_U$ are fixed and $\| \bv_U \|_{2}$ are provided as inputs.  Consider the inner maximization over $Q$ in isolation.  In this setting, the decision variable is the value of $\mathbb{P}[Q = \bv_U]$.  The inner optimization instance is a {\em linear program} as the objective is linear, and the constraint set -- defined with respect to a suitable optimal transportation cost -- can be expressed as the solution of a linear program, specified in the following:
\begin{equation} \label{eq:dro_innermax}
\underset{\boldsymbol{q}, \pi}{\max} ~ \langle \boldsymbol{q}, \mathbf{t} \rangle \quad \mathrm{s.t.} \quad \langle C, \pi \rangle \leq \epsilon, \pi \mathbf{1} = \boldsymbol{p}, \pi^T \mathbf{1} = \boldsymbol{q}, \pi \geq 0.
\end{equation}
Here, the matrix $C:=C(\bx,\by)$ models the cost of moving unit mass from point $\bx$ to $\by$, while $\mathbf{t}$ is the vector whose entries are $\|\bv_U\|_2 t_U$.  By recalling that strong duality holds for linear programs, we conclude that \eqref{eq:dro_innermax} is equivalent to the following:
\begin{equation} \label{eq:innerprob_LP}
\underset{\boldsymbol \lambda, s}{\min} ~ s \epsilon + \langle \boldsymbol{p}, \boldsymbol \lambda \rangle \qquad \mathrm{s.t.} \qquad s C + \boldsymbol \lambda \mathbf{1}^T \geq \mathbf{1} \mathbf{t}^T, s \geq 0.
\end{equation}
Now notice that the objective and all of the constraints, with the exception of the volume constraint, are $1$-homogeneous.  In particular, this means that the constraint $\sum t_U^{-d} \leq 1$ holds with equality at optimality.  Finally, by taking the size of the discretization $U$ to $0$ with respect to its (surface) volume, we recover the following
\begin{equation} \label{eq:innerprob_cts}
\underset{s, \lambda \in L1 (\mathrm{d}P)}{\mathrm{argmin}} ~~ s \epsilon + \int \lambda(\bx) \mathrm{d} P (\bx) \quad \mathrm{s.t.} \quad s C(\bx,\by) + \lambda (\bx) \geq \| \by \|_{K} , s \geq 0.
\end{equation} While the above proof sketch provides intuition for how one arrives at the result, we formally prove the Theorem here.
\begin{proof}[Proof of Theorem \ref{thm:dro_formulation}]
The formal proof of this result exploits standard results in the DRO literature. First, note that the optimization formulation \eqref{eq:dro-wass} can be equivalently stated with the relaxed constraint $\mathrm{vol}(K) \leq 1$ since for any $K$ with $\mathrm{vol}(K) < 1$, the objective can be decreased by considering $cK$ for $c > 1$ since $\|\cdot\|_{cK} = \frac{1}{c}\|\cdot\|_K$. For the form of the inner maximization problem, fix any feasible star body $K$. Note that since $K$ is a star body, we have $r:=\inf_{\bu \in \sd} \rho_K(\bu) > 0$ and for such an $r$, $rB^d \subseteq K$ so that $\|\bx\|_K \leq \frac{1}{r}\|\bx\|_2$. Moreover, by assumption $\mathbb{E}_P[\|\bx\|_2] < \infty$ which implies $\mathbb{E}_P[\|\bx\|_K] \leq \frac{1}{r}\mathbb{E}_P[\|\bx\|_2] < \infty$, so we conclude $\|\cdot\|_K \in L1(\mathrm{d}P)$. Thus, the assumptions of Theorem 1 in \cite{blanchet2019quantifying} are met, which states that the inner maximization problem can be written as \begin{align*}
    \sup_{Q:d_W(P,Q)\leq \epsilon} \mathbb{E}_Q[\|\bx\|_K] & = \inf\left\{s\epsilon  + \int \lambda(\bx)\mathrm{d}P(\bx) : (s,\lambda) \in \Lambda_{C,\|\cdot\|_K}\right\}
\end{align*} where the feasible set $\Lambda_{C,\|\cdot\|_K}$ is defined as \begin{align*}
\Lambda_{C,\|\cdot\|_K} := \left\{(s,\lambda) : s \geq 0,\ \lambda \in L1(\mathrm{d}P),\ \lambda(\bx) + sC(\bx,\by) \geq \|\by\|_K, \forall (\bx,\by)\right\}.    
\end{align*} Recognizing \eqref{eq:innerprob_cts} for fixed $K$ and minimizing over feasible $K$ yields \eqref{eq:dro_cts}.
\end{proof}

\subsection{Numerical illustrations} \label{sec:numerics-dro}

Using the formulation derived in Theorem \ref{thm:dro_formulation}, we now illustrate the effect of the robustness parameter and cost choices through two examples.  These examples were computing using \eqref{eq:innerprob_LP}. We will focus on visualizing these regularizers in $2$-dimensions for illustrative purposes.

\begin{figure}[h] 
\centering
\includegraphics[width=0.24\textwidth]{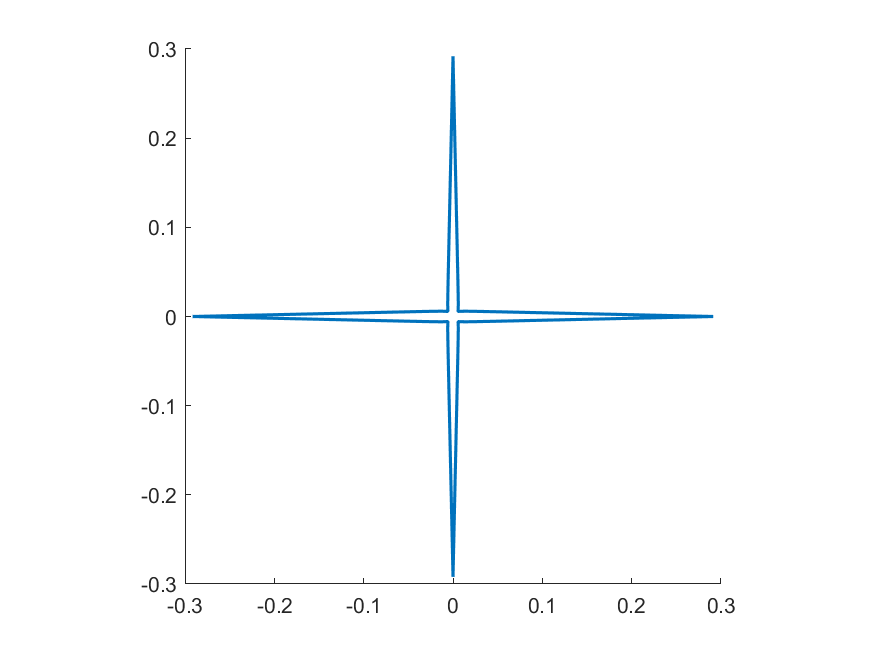}
\includegraphics[width=0.24\textwidth]{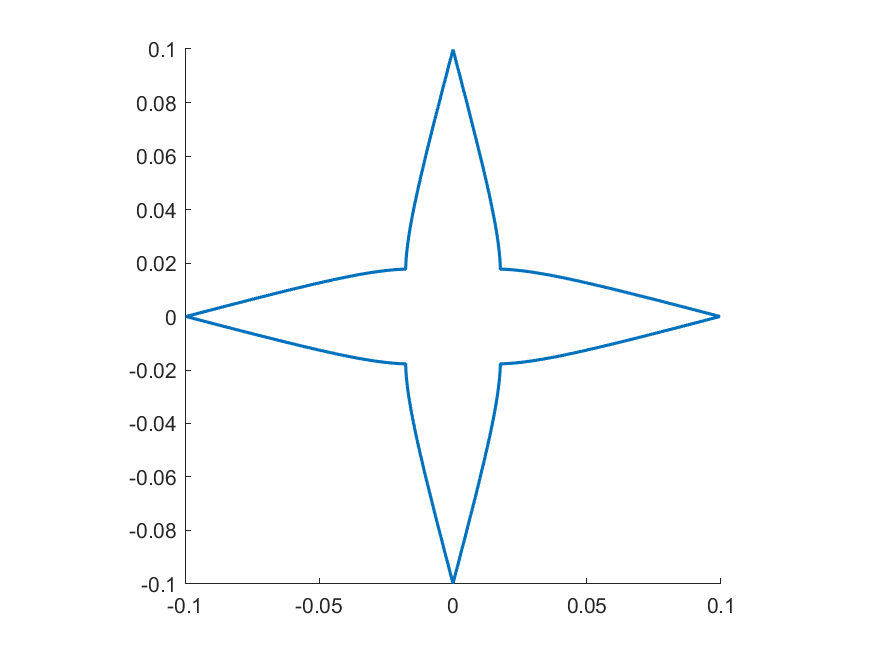}
\includegraphics[width=0.24\textwidth]{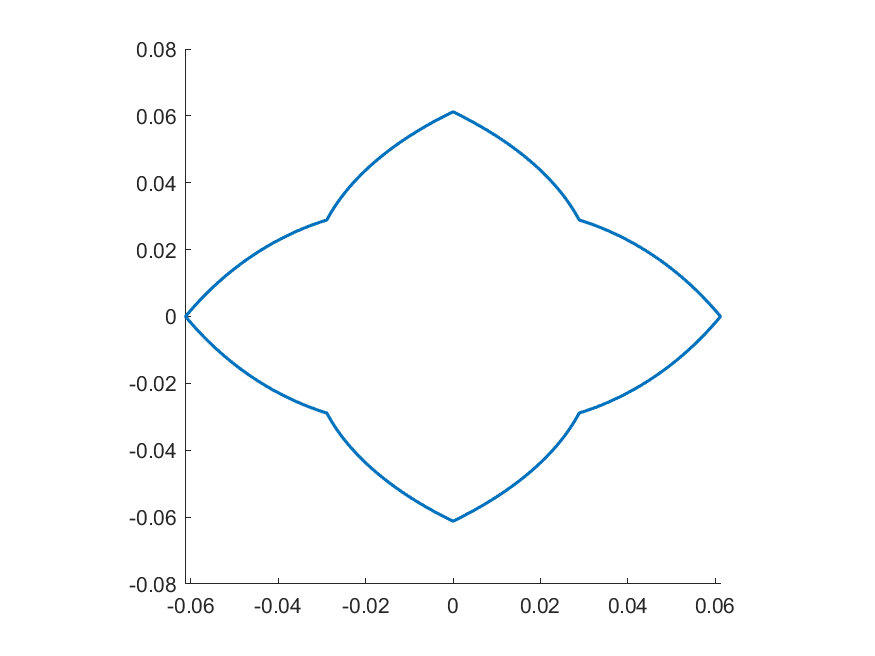}
\includegraphics[width=0.24\textwidth]{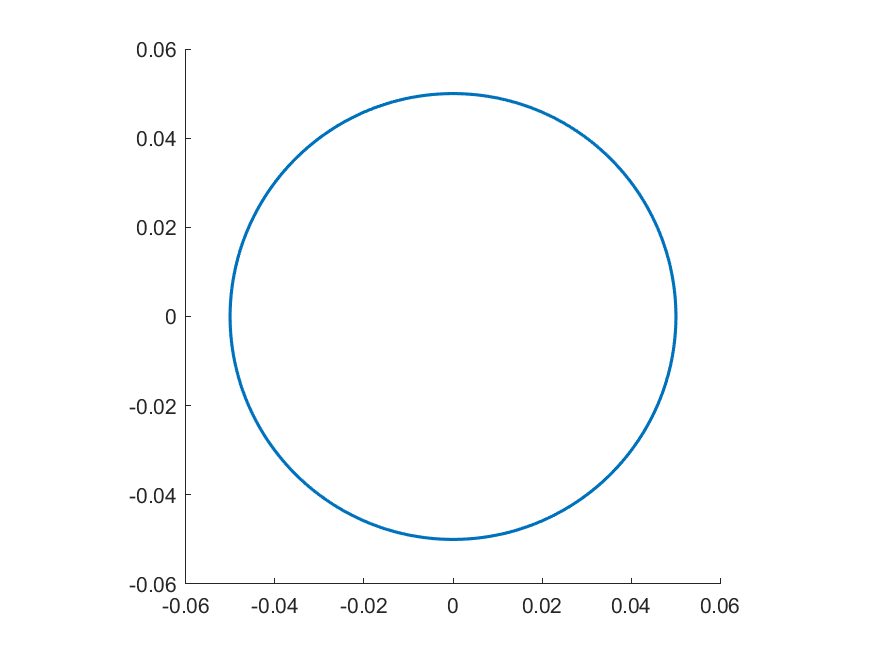}
\caption{Distributionally robust optimal regularizer for data supported on standard basis vectors.  The choice of $\epsilon$, from left to right, is $0.01$, $0.1$, $0.2$, $0.3$, with the cost given by the absolute distance.}
\label{fig:dro_L0}
\end{figure}

\paragraph{Example 1: Absolute cost distance.} In the first example, we consider a data distribution supported on the standard basis vectors and their negations $\{ (0,1), (-1,0), (0,-1), (1,0) \}$ with equal probability. In Figure \ref{fig:dro_L0} we show the distributionally robust optimal regularizer obtained via the formulation in \eqref{eq:dro_cts}.  The choices of $\epsilon$, from left to right, are $0.01$, $0.1$, $0.2$, $0.3$, while the cost function is the absolute distance of the argument $|\theta_i - \theta_j|$ (i.e. the arc length). For small values, we notice that the level resembles the $\ell_0$-norm, which is in effect placing dirac $\delta$-spikes on the standard basis vectors. As we increase $\epsilon$, the spikes broaden. We expect this because the optimal regularizer guards against distributions that are close to the original distribution in the Wasserstein-1 distance. At about $\epsilon \geq 0.3$, we see that the optimal regularizer is close to the $\ell_2$-norm -- this is consistent with an earlier remark that the optimal regularizer to \eqref{eq:dro_cts} is the $\ell_2$-norm for large $\epsilon$.

\begin{figure}[h] 
\centering
\includegraphics[width=0.24\textwidth]{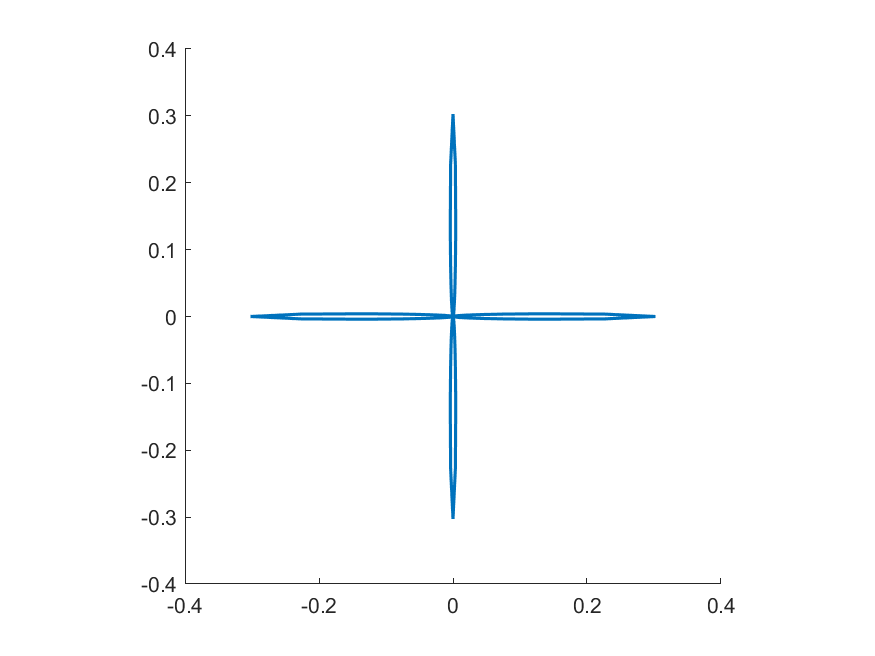}
\includegraphics[width=0.24\textwidth]{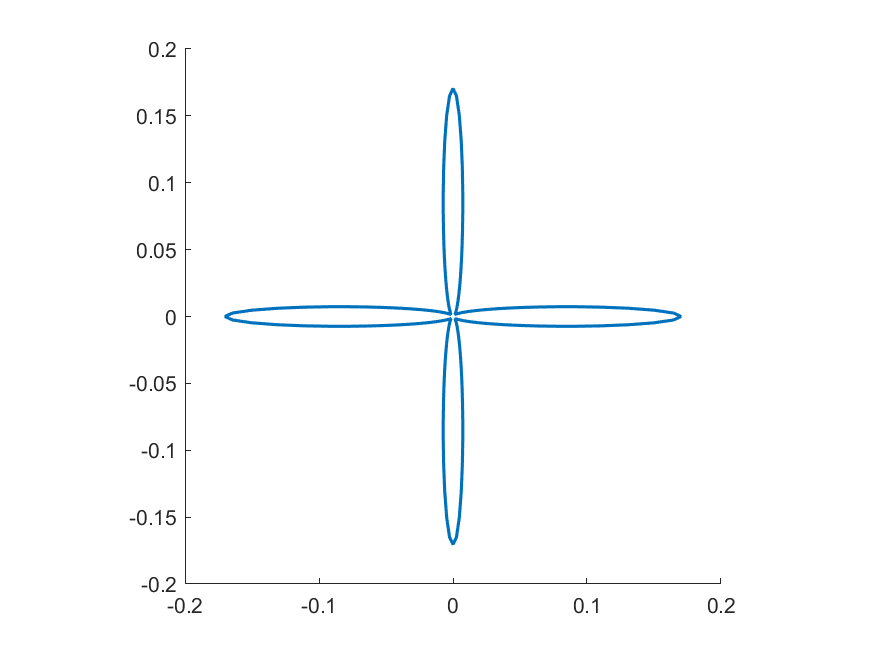}
\includegraphics[width=0.24\textwidth]{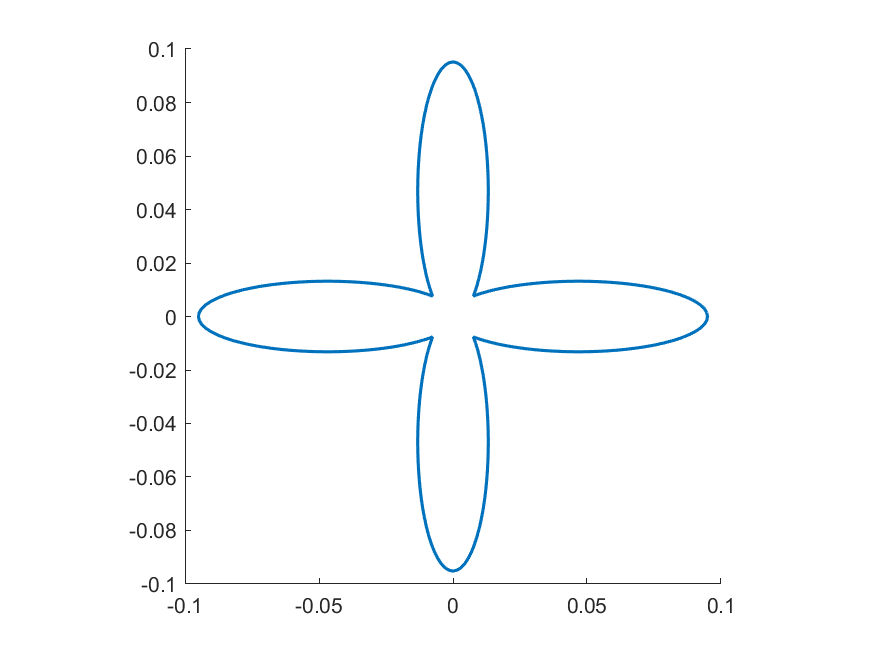}
\includegraphics[width=0.24\textwidth]{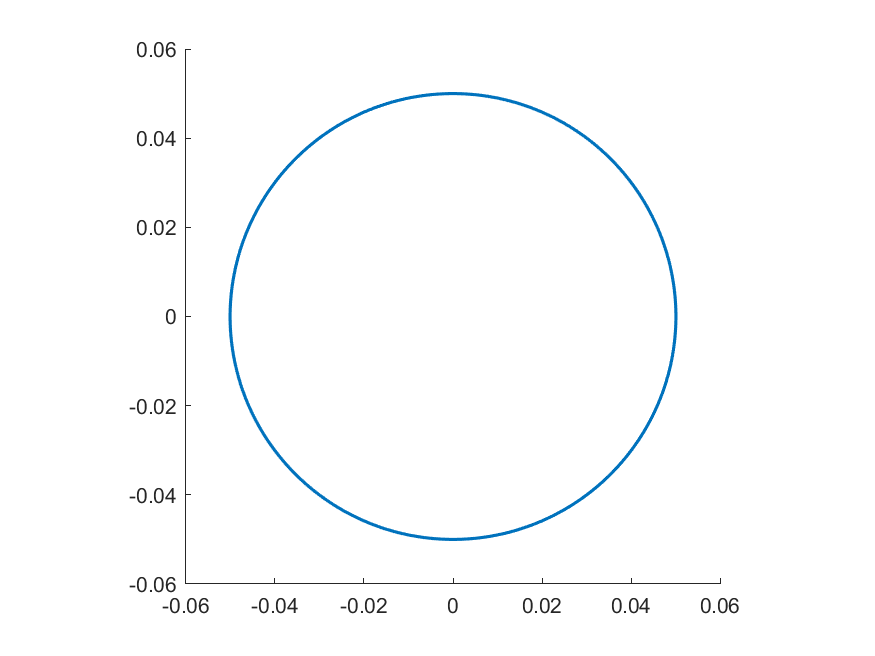}
\caption{Distributionally robust optimal regularizer for data supported on standard basis vectors.  The choice of $\epsilon$, from left to right, is $0.01$, $0.1$, $1.0$, $10$, and the cost function is the $\ell_2^2$ distance.}
\label{fig:dro_L0L1}
\end{figure}

\paragraph{Example 2: Quadratic cost.} In the second example we consider same data distribution, but with a quadratic $\ell_2^2$ cost function $(\theta_i - \theta_j)^2$.  The choices of $\epsilon$, from left to right, are $0.01$, $0.1$, $1.0$, $10$. We again observe dirac $\delta$-like structures at small $\epsilon$ that widen as $\epsilon$ grows, but the geometry of the level sets changes to exhibit smooth structure. In particular, in the previous example, we notice that the level set is ``spiky'' at $\theta = 0$ (the normal cone is non-trivial), whereas in the current example the level set is smooth (the normal cone is trivial).  A second difference is that the ``arms'' of the level set in the setting where the cost is $|\theta_i - \theta_j|$ grow wider as we go towards the center, whereas ``arms'' in the setting where the cost is $(\theta_i - \theta_j)^2$ grow more narrow as we go towards the center.  The difference comes from the fact that the squared L2-loss $(\theta_i - \theta_j)^2$ penalizes large deviations more heavily that the L1-cost $|\theta_i - \theta_j|$.  

These examples confirm that the robustness parameter $\epsilon$ systematically interpolates between highly data-adapted regularizers and isotropic norms, and that the choice of cost $C$ significantly influences the resulting geometry. We will discuss these topics from a more mathematical perspective in the next Section.

\subsection{Structural Properties of DRO Regularizers} \label{sec:structural-properties}

A natural question about the DRO formulation \eqref{eq:dro-wass} is how do the cost function $C$ and robustness parameter $\epsilon$ play a role in determining the geometry of the optimal regularizer. We illustrated in the previous section how these parameters influence the solution via numerical examples. We aim to develop a more mathematical understanding in the subsequent sections.

\subsubsection{The role of $\epsilon$ and its connection to uniform priors}

We will explore the role of the robustness parameter $\epsilon$ in this section. In order to understand the effect of $\epsilon$, it is instructive to analyze the two possible extremes:

\paragraph{Small robustness parameter $\epsilon$.} First, let's consider the extreme where $\epsilon \rightarrow 0$.  We show how in this regime, we essentially recover the original formulation  \eqref{eq:opt_star_formulation}. Note that the effect of $\epsilon \rightarrow 0$ is that the optimal choice of $s \rightarrow \infty$.  Recall the following constraint
\begin{equation} \label{eq:epsilon_uniform_constraint}
s C(\bx,\by) + \lambda (\bx) \geq \| \by \|_{K}.  
\end{equation}
Whenever $\bx \neq \by$, one has $C(\bx,\by) > 0$.  Because one has $s \rightarrow \infty$, the constraint \eqref{eq:epsilon_uniform_constraint} will be satisfied eventually.  This leaves the case where $\bx = \by$.  Note that it is necessary to adopt the convention $0 \times \infty = 0$ in what follows.  The constraint \eqref{eq:epsilon_uniform_constraint} then translates to $\lambda(\bx) \geq \| \bx \|_{K}$ for all $\bx$.  In other words, the objective reduces to $\mathbb{E} [ \| \bx \|_{K} ]$, as we expect.  

\paragraph{Large robustness parameter $\epsilon$.} Second, let's consider the extreme where $\epsilon \rightarrow \infty$.  Then the objective drives $s \rightarrow 0$, in which case the inequality \eqref{eq:epsilon_uniform_constraint} reduces to $\lambda (\bx) \geq \| \by \|_{K}$.  This means that $\| \by \|_{K}$ is to be uniformly bounded by some constant.  By pushing $\lambda \rightarrow 0$, we encourage the volume of $K$ to be as large as possible, so in fact the gauge evaluation is a constant -- that is, $K$ is the (scaled) unit sphere.  This may be expected -- when $\epsilon$ is large, one has to guard against the worst possible distribution, and that has zero relation to the base distribution on which the data is drawn from.  When there is no prior, one simply selects a regularizer that is uniform across all directions; i.e., the {\em uniform} prior.  

Thus, increasing $\epsilon$ transitions the optimal regularizer from a data-dependent geometry to the isotropic $\ell_2$-ball. Distributional robustness therefore plays a role analogous to imposing a uniform prior, with $\epsilon$ controlling the tradeoff. This intuition aligns with the numerical illustrations we discussed in Section \ref{sec:numerics-dro}.

\subsubsection{Homogeneity and normalization properties}

Next, we describe a number of basic properties regarding \eqref{eq:dro_cts}.

First, let $K$, $\lambda$, and $s$ be feasible in \eqref{eq:dro_cts}.  Suppose $\mathrm{vol}(K) < 1$.  Let $c > 1$ be such that $\mathrm{vol}(c K) = 1$.  Then $\| \by \|_{c K} = \|\by \|_{K} / c$.  Now notice that the objective and the constraints are $1$-homogeneous.  In particular, the triplet $(cK,\lambda/s, s/c)$ is also feasible, but by doing so, we decrease the objective by a factor of $1/c < 1$.  

Second, notice from the constraint in \eqref{eq:dro_cts} that one has $\lambda (\bx) \geq \sup_{\by} \| \by \|_{K} - s C(\bx,\by)$.  Now, because $\mathrm{d}P$ is a positive measure, at optimality we would in fact have
\begin{equation*} 
\lambda (\bx) ~=~ \sup_{\by} \, \| \by \|_{K} - s C(\bx,\by).
\end{equation*}
Indeed, this is shown to be a consequence of Theorem 1 in \cite{blanchet2019quantifying} (see equation (9) and the discussion surrounding it). This means that the function $\lambda$ in \eqref{eq:dro_cts} can be expressed entirely in terms of the set $K$, $s$, and the cost $C$.

Third, we make a similar characterization regarding $\| \by \|_{\hat{K}}$, where $\hat{K}$ is the optimal solution to \eqref{eq:dro_cts}.  Let $K$, $\lambda$, and $s$ be feasible in \eqref{eq:dro_cts}.  From the constraints we have 
$$\| \by \|_{K} \leq \inf_{\bx} s C(\bx,\by) + \lambda (\bx).$$
There is a sense in which equality should also hold for $\hat{K}$ in the above inequality; however, it is not a priori clear if the expression on the right-hand side $\inf_{\bx} s C(\bx,\by) + \lambda (\bx)$, as it is defined above, necessarily specifies a function that is $1$-homogeneous and, hence, realizable as the gauge function of a star body. However, there is an important case where this is true -- this is when $C$ is a norm.


\color{black}

\begin{proposition} \label{prop:lambda-phi-characterization}
    Let $K$ be a star body and suppose $C(\bx,\by) = \|\bx-\by\|$, where $\|\cdot\|$ is any norm. Define $$s_*:= \sup_{\|\by\|=1}\|\by\|_K \in (0,\infty).$$ Consider the function $\lambda(\bx) = \sup_{\by} \|\by\|_K - s\|\bx-\by\|$ and $\phi(\by):= \inf_{\bx} s\|\bx-\by\| + \lambda(\bx)$. Then \begin{itemize}
        \item if $s < s_*$, $\lambda(\bx) = +\infty$ for all $\bx$,
        \item if $s \geq s_*$, we have that $\lambda(\bx) = \phi(\bx)$ for all $\bx$. Moreover, $\phi$ (and hence $\lambda$) is $1$-homogeneous, continuous, and positive over the unit sphere, satisfying the bounds $\|\bx\|_K \leq \lambda(\bx)=\phi(\bx)\leq s\|\bx\|$.
    \end{itemize}
\end{proposition}

\begin{proof}[Proof of Proposition \ref{prop:lambda-phi-characterization}]
    First, note that $s_*$ is positive and finite since $\by \mapsto \|\by\|_K$ is continuous (since $K$ is a star body) over the compact set $\{\by : \|\by\|=1\}$. For $s < s_*$, take $\bv$ with $\|\bv\|=1$ and $s_* \geq \|\bv\|_K > s$ (which exists since $\{\bv : \|\bv\|=1\}$ is compact and $\|\cdot\|_K$ is continuous). Then note that for any $\bx$, \begin{align*}
        \lambda(\bx) \geq \|t\bv\|_K - s\|\bx-t\bv\| & \geq t\|\bv\|_K - s(\|\bx\|+t) \\
        & = t(\|\bv\|_K - s) - s\|\bx\|.
    \end{align*} Since $\|\bv\|_K -s > 0$, taking $t \rightarrow \infty$ shows that $\lambda(\bx) = +\infty$.

    For $s \geq s_*$, we first show $\lambda(\bx) = \phi(\bx)$. Note that trivially $\phi \leq \lambda$ since $$\phi(\bx) \leq s\|\bx - \bx\| + \lambda(\bx) =\lambda(\bx).$$ To show $\phi \geq \lambda$, note that for any $\bx,\by,\bz$, $$s\|\bx-\by\| + \lambda(\bx) \geq s\|\bx-\by\| + \|\bz\|_K - s\|\bx-\bz\| \geq -s\|\by-\bz\| + \|\bz\|_K$$ where the last line follows from the triangle inequality. Taking the infimum of the left hand side and the supremum of the right hand side yields $$\phi(\by)=\inf_{\bx} s\|\bx-\by\|+\lambda(\bx) \geq \sup_z\|\bz\|_K - s\|\by-\bz\| = \lambda(\by).$$ Hence $\phi(\by) = \lambda(\by).$ 
    
    We now show that $\phi=\lambda$ satisfies the conditions to be the gauge of a star body. For homogeneity, note that for $t \geq 0$, $$\lambda(t\bx) = \sup_{\bz}\|\bz\|_K - s\|t\bx-\bz\| = \sup_{\bz t}\|\bz/t\|_K - ts\|\bx-\bz/t\|=t\sup_{\tilde{\bz}}\|\tilde{\bz}\|_K - s\|\bx-\tilde{\bz}\|=t\lambda(\bx).$$ For continuity, the proof of Lemma \ref{lem:dro-lipschitz} in Section \ref{sec:existence-of-minimizers} establishes Lipschitz continuity. Finally, for positivity, note that for any $\bu \in \sd$, we have \begin{align*}
        0 < \|\bu\|_K =\|\bu\|_K-s\|\bu-\bu\| \leq \sup_{\by}\|\by\|_K-s\|\bu-\by\|=\lambda(\bu)=\phi(\bu).
    \end{align*} Moreover, $\phi=\lambda$ is always finite because of the following: since $s \geq s_* = \sup_{\|\by\|=1}\|\by\|_K$, we have that for any $\by$, $\|\by\|_K = \|\by\|\|\by/\|\by\|\|_K \leq \|\by\|s_*$ so for $s \geq s_*$, \begin{align*}
        \|\by\|_K - s\|\by\| \leq (s_*-s)\|\by\|\leq 0
    \end{align*} and at $\by = 0$, the upper bound holds with equality so in fact \begin{align}
    \sup_{\by}\|\by\|_K - s\|\by\|=0\ \text{for}\ s \geq s_*. \label{eq:untranslated-sup}
    \end{align} We now translate this to finiteness of $\lambda$. In particular, note that for all $\bx,\by$, the reverse triangle inequality $\|\bx-\by\| \geq \|\by\|-\|\bx\|$ gives  \begin{align*}
        \|\by\|_K - s\|\bx-\by\| \leq \|\by\|_K - s\|\by\| + s\|\bx\|.
    \end{align*} Taking the supremum over $\by$ and using \eqref{eq:untranslated-sup} gives the following bound for all $\bx \in \R^d$ when $s \geq s_*$: \begin{align*}
        \lambda(\bx)=\sup_{\by} \|\by\|_K - s\|\bx-\by\| \leq \sup_{\by} \|\by\|_K-s\|\by\| + s\|\bx\| = s\|\bx\| < \infty.
    \end{align*}
\end{proof}

\subsubsection{Lipschitz penalization induced by Wasserstein-1 distance}

While the previous sections give general intuition for how the parameters influence properties of the optimal solution, we derive more specific geometric properties here by considering the case when the underlying cost function is the unsquared Euclidean distance $C(\bx, \by) := \|\bx - \by\|_2$. Note that this precisely gives rise to the ambiguity set induced by the Wasserstein-1 distance. In this case, an explicit characterization of the inner maximization instance in \eqref{eq:dro-wass} can be obtained using known duality results. In particular, the following result shows that using the Wasserstein-1 distance explicitly penalizes the Lipschitz constant $\mathrm{Lip}(\|\cdot\|_K) = \mathrm{Lip}(K)$ of the optimal regularizer, hence robustifying it by ensuring it is less sensitive to small perturbations as the robustness parameter $\epsilon$ grows. For simplicity of the proof, we will show this for star bodies with well-behaved kernels.

Prior to the proof, we remark that Lipschitz continuity for star body gauges is equivalent to the geometric property that their kernels contain a Euclidean ball. For example, as shown Proposition 2 of \cite{leong2025optimal}, if there exists an $r > 0$ such that $rB^d \subseteq \mathrm{ker}(K)$, then $\|\cdot\|_K$ is $1/r$-Lipschitz. The Lipschitz constant of $\|\cdot\|_K$ corresponds to taking the inverse of the largest Euclidean ball that lies in the kernel of $K$: $\mathrm{Lip}(K):=\inf\{1/r : rB^d \subseteq \mathrm{ker}(K)\}< \infty.$

\begin{proposition} \label{prop:lipschitz-regularization}
    Let $C(\bx,\by) = \| \bx-\by \|_2$.  Suppose $K$ is a star body such that $r_{\mathrm{in}}=r_{\mathrm{ker}}>0$ where $r_{\mathrm{in}}:=\inf_{\bu \in \sd}\rho_K(\bu)$ and $r_{\mathrm{ker}}:=\sup\{r >0:rB^d\subseteq\mathrm{ker}(K)\}$. Then for any $\epsilon \geq 0$, the inner maximization problem to \eqref{eq:dro-wass} with $d_W=W_1$ becomes \begin{align*}
        \max_{Q:d_W(P,Q)\leq \epsilon} \mathbb{E}_Q[\|\bx\|_K] = \mathbb{E}_P[\|\bx\|_K] + \epsilon \cdot \mathrm{Lip}(K).
    \end{align*}
\end{proposition}
\begin{proof}[Proof of Proposition \ref{prop:lipschitz-regularization}]
    By assumption, note that $\mathrm{Lip}(K) = 1/r_{\mathrm{ker}}$. For the form of the maximal objective, note that Theorem 7 in \cite{kuhn2019wasserstein} with $p = 1$ (or equation (9) in \cite{blanchet2019quantifying}) shows that \begin{align*}
         \max_{Q:W_1(P,Q)\leq \epsilon}  \mathbb{E}_Q[\|\bx\|_K] = \inf_{s \geq 0}\left\{\mathbb{E}_P\left[\sup_{\mathbf{z}}  \|\mathbf{z}\|_K - s\|\mathbf{z} - \bx\|_2\right] +  \epsilon \cdot s\right\}.
    \end{align*} We claim that $$\sup_{\mathbf{z}}  \|\mathbf{z}\|_K - s\|\mathbf{z} - \bx\|_2 = \begin{cases}
    \|\bx\|_K &\ \text{if}\ s \geq \mathrm{Lip}(K) \\
    +\infty &\ \text{if}\ 0 \leq s < \mathrm{Lip}(K).
    \end{cases}$$ Suppose $s \geq \mathrm{Lip}(K)$. Then we have that since $\|\cdot\|_K$ is Lipschitz, \begin{align*}
        \|\mathbf{z}\|_K - s\|\mathbf{z}-\bx\|_2 & = \|\bx\|_K + (\|\mathbf{z}\|_K-\|\bx\|_K) - s\|\mathbf{z}-\bx\|_2 \\
        & \leq \|\bx\|_K + \underbrace{\left(\mathrm{Lip}(K)-s\right)}_{\leq 0}\|\mathbf{z}-\bx\|_2 \\
        & \leq \|\bx\|_K.
    \end{align*} Taking the supremum on the left-hand side yields $\sup_{\mathbf{z}}\|\mathbf{z}\|_K - s\|\mathbf{z}-\bx\|_2 \leq \|\bx\|_2$ with equality when $\mathbf{z}=\bx$, so $\sup_{\mathbf{z}}\|\mathbf{z}\|_K - s\|\mathbf{z}-\bx\|_2  = \|\bx\|_K$ with $s \geq \mathrm{Lip}(K)$. Now consider the case $s < \mathrm{Lip}(K)$. Note that since $\|\cdot\|_K$ is Lipschitz and only vanishes at the origin, we have that $\|\mathbf{z}\|_K \leq \mathrm{Lip}(K)\|\mathbf{z}\|_2$ for any $\mathbf{z} \in \R^d$ so $\max_{\|\mathbf{z}\|_2=1}\|\mathbf{z}\|_K\leq\mathrm{Lip}(K)$. But in fact, this holds with equality since $$\max_{\|\mathbf{z}\|_2=1}\|\mathbf{z}\|_K=\max_{\|\mathbf{z}\|_2=1}\frac{1}{\rho_K(\mathbf{z})} = \frac{1}{\min_{\|\mathbf{z}\|_2=1}\rho_K(\mathbf{z})}= \frac{1}{r_{\mathrm{in}}}= \frac{1}{r_{\mathrm{ker}}}=\mathrm{Lip}(K).$$ By continuity, there must exist a $\bu \in \sd$ such that $\mathrm{Lip}(K)\geq\|\bu\|_K > s.$ For such a direction $\bu$, consider positive scalings $r \geq 0$: \begin{align*}
        \|r\bu\|_K - s\|r\bu - \bx\|_2 & \geq r\|\bu\|_K - s(r\|\bu\|_2 +\|\bx\|_2) \\ 
        & = r \underbrace{(\|\bu\|_K - s)}_{> 0} -s\|\bx\|_2 \\
        & \longrightarrow+\infty\ \text{as}\ r \longrightarrow +\infty.
    \end{align*} Hence we must have $\sup_{\mathbf{z}}\|\mathbf{z}\|_K - s\|\mathbf{z}-\bx\|_2 = +\infty$ when $0 \leq s < \mathrm{Lip}(K).$


    Combining our two cases, we see that \begin{align*}
        \max_{Q:W_1(P,Q)\leq \epsilon}  \mathbb{E}_Q[\|\bx\|_K] & = \inf_{s \geq 0}\left\{\mathbb{E}_P\left[\sup_{\mathbf{z}}  \|\mathbf{z}\|_K - s\|\mathbf{z} - \bx\|_2\right] + \epsilon \cdot s\right\} \\
        & = \inf_{s \geq \mathrm{Lip}(K)}\mathbb{E}_P\left[\|\bx\|_K\right] + \epsilon \cdot s \\
        & = \mathbb{E}_P\left[\|\bx\|_K\right] + \epsilon \cdot \mathrm{Lip}(K)
    \end{align*}
\end{proof} 

\begin{remark} While we assume that the parameters $r_{\mathrm{in}}$ and $r_{\mathrm{ker}}$ are equal to one another in this proof, we believe it may be possible to extend this result to the case when $r_{\mathrm{in}} > r_{\mathrm{ker}}$. Note that star bodies in general have $r_{\mathrm{in}} > r_{\mathrm{ker}}$ since the kernel of a star body may be trivial while still containing a Euclidean ball (consider, e.g., any $\ell_q$-quasinorm unit ball for $q \in (0,1)$). \end{remark}

The additional Lipschitz penalization provides an interesting intuition for how distributional robustness naturally induces regularity in the optimal regularizer. However, such a penalization makes it challenging to give a precise characterization of minimizers for the DRO problem in general over all star bodies. For star bodies that satisfy the assumptions of Proposition \ref{prop:lipschitz-regularization}, we present a result towards a possible characterization through the use of dual mixed volumes. In particular, consider the set of star bodies $\tilde{\mathcal{S}}$ such that $r_{\mathrm{in}} = r_\mathrm{ker}$. The following Proposition shows that the inner maximization problem can be written as the supremum of a dual mixed volume functional that depends on $K \in \tilde{\mathcal{S}}$ and a particular star body $W^\epsilon_S$ that is a ``radial'' $\epsilon$-combination of a data-dependent star body $L_P$ and an arbitrary star body $S$:

\begin{proposition}\label{prop:convex-lipschitz-prop}
    Fix $\epsilon \geq 0$. Let $P$ be a distribution with $\mathbb{E}_P[\|\bx\|_2] < \infty$ that admits a density $p$ with respect to the Lebesgue measure such that the function $\rho_P$ in equation \eqref{eq:rhoP} is positive and continuous over the unit sphere. For any $K \in \tilde{\mathcal{S}},$ denote \begin{align*}
        J_{\epsilon}(K) := \mathbb{E}_P[\|\bx\|_K] + \epsilon\cdot \mathrm{Lip}(K).
    \end{align*} Then we have the following: \begin{enumerate}
        \item The dual mixed volume representation holds \begin{align}
            J_{\epsilon}(K) := \sup_{S \in \mathcal{S}_1} \tilde{V}_{-1}(K,W^\epsilon_S) \label{eq:J_eps_rep}
        \end{align} where $\mathcal{S}_1:=\{S\ \text{star body} : M_{d+1}(S)=1\}$ with $M_{d+1}(S) = \frac{1}{d}\int_{\sd}\rho_S(\bu)^{d+1}d\bu$ and $W^\epsilon_S$ is the star body with radial function $\rho_{W^\epsilon_S}$ defined by $$\rho_{W^\epsilon_S}^{d+1}(\bu) := d\rho_{P}(\bu)^{d+1} + \epsilon \rho_{S}(\bu)^{d+1},\ \quad \bu \in \sd.$$ In particular, $W^\epsilon_S$ is the $(d+1)$-harmonic radial combination \cite{Lutwak96} between $d^{1/(d+1)}L_P$ and $\epsilon^{1/(d+1)}S$.
        \item We have the following lower bound on the objective over $\tilde{\mathcal{S}}$, $$\inf_{K \in \tilde{\mathcal{S}},\ \mathrm{vol}(K)=1} J_{\epsilon}(K) \geq \sup_{S \in \mathcal{S}_1} \mathrm{vol}(W^\epsilon_S)^{\frac{d+1}{d}}.$$ 
    \end{enumerate}
\end{proposition}

\begin{proof}[Proof of Proposition \ref{prop:convex-lipschitz-prop}]
    Note that for convex bodies $K$, we have that the Lipschitz constant satisfies $\mathrm{Lip}(K) = \sup_{\|\bu\|_2=1}\|\bu\|_K$. Additionally, note that $$\tilde{V}_{-1}(K,S) = \frac{1}{d}\int_{\sd}\rho_S^{d+1}(\bu)\rho_K(\bu)^{-1}\mathrm{d}\bu = \int_{\sd}\|\bu\|_K\mathrm{d}\mu_S(\bu)$$ where $\mu_S$ is the probability measure on the sphere $\sd$ with density $1/d \cdot \rho_S^{d+1}$ with respect to the surface measure $\mathrm{d}\bu$. Note that this is indeed a probability measure when restricting ourselves to $S \in \mathcal{S}_1:=\{S \in\mathcal{S}^d: M_{d+1}(S)=1\}.$ One can show that the space of measures $\mathcal{P}_{\mathcal{S}_1}:=\{\mu_S :S \in \mathcal{S}_1,\ \mathrm{d}\mu_S(\bu) = d^{-1}\rho_S^{d+1}(\bu)\mathrm{d}\bu\}$ is weak-* dense on the space of all Borel probability measures on the sphere $\mathcal{P}(\sd)$ since it is equivalent to the set of all measures with strictly positive, continuous densities $\mathcal{P}_{\mathrm{cont}} :=\{\mu :\mathrm{d}\mu(\bu)=f(\bu)\mathrm{d}\bu,\ f > 0,f\ \text{continuous},\int_{\sd}f=1\}$. That is to say, for every $\mu \in \mathcal{P}(\sd)$, there exists a sequence $(\mu_k) \subset \mathcal{P}_{\mathrm{cont}}$ such that $\int_{\sd} g \mathrm{d}\mu_k \rightarrow \int_{\sd}g \mathrm{d}\mu$ for $g$ continuous on the sphere $\sd$ as $k \rightarrow \infty$. This implies that \begin{align*} \sup_{S \in \mathcal{S}_1} \tilde{V}_{-1}(K,S) & = \sup_{\mu_S \in \mathcal{P}_{\mathcal{S}_1}} \int_{\sd}\|\bu\|_K\mathrm{d}\mu_S(\bu) \\
    & = \sup_{\mu \in \mathcal{P}_{\mathrm{cont}}} \int_{\sd}\|\bu\|_K\mathrm{d}\mu(\bu) \\
    &=\sup_{\mu \in \mathcal{P}(\sd)} \int_{\sd}\|\bu\|_K\mathrm{d}\mu(\bu).
    \end{align*} Moreover, we have that $$\sup_{S \in \mathcal{S}_1} \tilde{V}_{-1}(K,S) =\sup_{\mu \in \mathcal{P}(\sd)} \int_{\sd}\|\bu\|_K\mathrm{d}\mu(\bu) = \sup_{u \in \sd} \|\bu\|_K = \mathrm{Lip}(K)$$ where the second equality follows by noting that for any $\mu \in \mathcal{P}(\sd)$, by continuity and compactness, $\|\cdot\|_K$ attains its maximum over $\sd$ for some $\bu_*$ so that $$\int_{\sd} \|\bu\|_K \mathrm{d}\mu(\bu)\leq \|\bu_*\|_K \cdot \int_{\sd}\mathrm{d}\mu(\bu)= \|\bu_*\|_K$$ and this inequality holds with equality at the dirac measure $\mu:=\delta_{\bu_*}.$ Hence we attain \begin{align*}
        J_{\epsilon}(K) & =  \mathbb{E}_P[\|\bx\|_K] + \epsilon\cdot \mathrm{Lip}(K)\\
        & = d\tilde{V}_{-1}(K,L_P) + \epsilon \sup_{S \in \mathcal{S}_1}\tilde{V}_{-1}(K,S) \\
        & = \sup_{S \in \mathcal{S}_1}\left\{d\tilde{V}_{-1}(K,L_P) + \epsilon \tilde{V}_{-1}(K,S)\right\} \\
        & =  \sup_{S \in \mathcal{S}_1} \tilde{V}_{-1}(K,W^\epsilon_S)
    \end{align*} where we used the definition of $W^\epsilon_S$ in the final equality. This gives the representation \eqref{eq:J_eps_rep}.

    For the final result, this is an application of Lutwak's inequality (see Theorem \ref{thm:lutwak-dmv} in Section \ref{sec:prelims}). Indeed, for any $S \in \mathcal{S}_1$ and $K \in \tilde{\mathcal{S}}$ with unit volume, we have with $L=W^\epsilon_S$ that $\tilde{V}_{-1}(K,W^\epsilon_S) \geq \mathrm{vol}(W^\epsilon_S)^{(d+1)/d}$. Taking the supremum over $S$ and infimum over unit-volume $K$ yields the desired bound. 
\end{proof}

The above result shows how one can view the inner Wasserstein-1 DRO objective as a dual mixed volume between a star body with a well-behaved kernel $K$ and a data- and $\epsilon$-dependent star body $W_S^{\epsilon}$. This new star body is a particular radial combination between the data-dependent star body $L_P$ and an arbitrary star body $S$ with normalized moment $M_{d+1}(S)$. With this view, it is possible to write the entire objective as the supremum of a single dual mixed volume involving $K$ and $W^\epsilon_S$ over $S$. We note that it is challenging in general to obtain an exact description of the maximizer of $\sup_{S \in \mathcal{S}_1} \mathrm{vol}(W_S^{\epsilon})^{(d+1)/d}$. If a maximizer $S_*$ exists and induces an  $W_{S_*}^{\epsilon} \in \tilde{\mathcal{S}}$, then we would indeed have that $K_*:=\mathrm{vol}(W_{S_*}^{\epsilon})^{-1/d}W_{S_*}^{\epsilon}$ is a minimizer to the $\tilde{\mathcal{S}}$-constrained DRO problem. For the general case of $K$ being a star body, this dual mixed volume representation may not hold and the true inner DRO objective may involve more than a simple Lipschitz penalization.

\subsubsection{Existence of minimizers} \label{sec:existence-of-minimizers}
Our next result concerns the existence of minimizers to \eqref{eq:dro-wass} and \eqref{eq:dro_cts}.  More precisely, our goal is to show that minimizers to \eqref{eq:dro-wass} and \eqref{eq:dro_cts} exist {\em for all} distributions $P$ so long as $\epsilon > 0$.  This is in sharp contrast with \eqref{eq:opt_star_formulation}, which corresponds to the case where $\epsilon = 0$. In that case, existence of minimizers was shown \cite{leong2025optimal} for general distributions when restricted to star bodies with a fixed size Euclidean ball in their kernels. Moreover, as we noted in the previous section, minimizers to \eqref{eq:opt_star_formulation} cannot exist for empirical measures. Here, we allow for general distributions $P$ and norm-based costs $C(\bx, \by) = \|\bx-\by\|.$

\begin{theorem} \label{thm:existence-dro-general-norm}
Consider \eqref{eq:dro_cts}.  Suppose that the cost function $C(\bx,\by) = \| \bx-\by \|$ for a general norm $\|\cdot\|$.  Suppose $\epsilon > 0$ and that the optimal value to \eqref{eq:dro_cts} is finite.  Then there exists a closed star $K$ that attains the minimum objective value.
\end{theorem}

To prove this, we first establish several helpful auxillary lemmas. The first Lemma shows that a particular useful functional is Lipschitz continuous.
\begin{lemma} \label{lem:dro-lipschitz}
Let $C(\bx,\by)=\|\bx-\by\|$ be a norm and set $s > 0$. Define the function $g(\by) := \inf_{\bx} s C(\bx,\by) + \lambda (\bx) $.  Then $|g(\by_1) - g(\by_2)| \leq s \| \by_1 - \by_2 \|$.  In particular, this means that $g$ is Lipschitz continuous with respect to $\|\cdot\|$. Lipschitz continuity with respect to $\|\cdot\|_2$ easily follows by equivalence of norms. 
\end{lemma}

\begin{proof} [Proof of Lemma \ref{lem:dro-lipschitz}]
Let $\bx^{\star}(\by)$ be the $\arg \min$ of $\bx$ in the definition of $g$. (If the $\arg \min$ is not unique, then make an arbitrary choice -- the proof does not depend on uniqueness.)  Recall from the triangle inequality one has $| \| \bx - \by_1 \| - \| \bx - \by_2 \| | \leq \| \by_1 - \by_2 \|$.  Then
\begin{equation*}
\begin{aligned}
g(\by_1) ~=~ & s C( \bx^{\star}(\by_1), \by_1) + \lambda (\bx^{\star}(\by_1)) \\
\geq~ & s C( \bx^{\star}(\by_1), \by_2) - s \| \by_2 - \by_1 \| + \lambda (\bx^{\star}(\by_1)) \\
\geq~ & g(\by_2) - s \| \by_2 - \by_1 \|. 
\end{aligned}
\end{equation*}
The first inequality follows from the triangle inequality and the second inequality follows from the definition of $g$.  Similarly, one has $g(\by_2) \geq g(\by_1) - s \| \by_2 - \by_1 \|$, which implies the result.
\end{proof}

Then we need to show a uniform bound on input points over the sphere induced by $\|\cdot\|$:

\begin{lemma} \label{lem:uniform-bound} Fix $s > 0$ and suppose $(K,s,\lambda)$ is feasible in \eqref{eq:dro_cts} with $\lambda = \lambda_{K,s}$ defined by $\lambda_{K,s}(\bx):=\sup_\by\|\by\|_K - s\|\bx-\by\|$ and $s \geq s_*:=s_*(K)$ defined in Proposition \ref{prop:lambda-phi-characterization}. Then $\sup_{\|\by\|=1}\|\by\|_K \leq \|\hat{\by}\|_K + 2s$ for any $\hat{\by}$ such that $\|\hat{\by}\|=1$. Thus $\|\cdot\|_K$ is uniformly bounded over the sphere $\{\bu : \|\bu\|=1\}$.
\end{lemma}
\begin{proof}[Proof of Lemma \ref{lem:uniform-bound}]
    Pick $\hat{\bx}$ with $\|\hat{\by}\|_K = s\|\hat{\bx}-\hat{\by}\| + \lambda(\hat{\bx})$. Then for any $\by$ with $\|\by\|=1$, we see that \begin{align*}
        \|\by\|_K \leq s\|\hat{\bx} - \by\| + \lambda(\hat{\bx}) \leq s\|\hat{\bx} - \hat{\by}\| + s\|\hat{\by}-\by\| + \lambda(\hat{\bx}) \leq \|\hat{\by}\|_K + 2s
    \end{align*} where we used the fact that both $\hat{\by}$ and $\by$ have unit $\|\cdot\|$-norm.
\end{proof}

Finally, we require showing that the objective functional is continuous with respect to the star body argument $K$.
\begin{lemma} \label{lem:cts_wrt_hausdorff}
Given a star body $K$ and measure $P$, define
\begin{equation*}
f(K) = \int \lambda_{K,s}(\bx) \mathrm{d} P (\bx) \qquad \text{where} \qquad \lambda_{K,s}(\bx) = \sup_{\by} \| \by \|_K -s C(\bx,\by).
\end{equation*}
We then have
\begin{equation*}
| f (K_1) - f (K_2) | \leq \| \|\cdot\|_{K_1} - \|\cdot\|_{K_2} \|_{\infty}.
\end{equation*}
\end{lemma}

\begin{proof} [Proof of Lemma \ref{lem:cts_wrt_hausdorff}]  First we have
\begin{equation*}
\big| \big( \|\by\|_{K_1} - s C(\bx,\by) \big) - \big( \|\by\|_{K_2} - s C(\bx,\by) \big) \big| \leq \| \|\cdot\|_{K_1} - \|\cdot\|_{K_2} \|_{\infty} = : c.
\end{equation*}
By taking the maximum over $\by$, one also has 
\begin{equation*}
\big| \big( \sup_{\by} \|\by\|_{K_1} - s C(\bx,\by) \big) - \big( \sup_{\by} \|\by\|_{K_2} - s C(\bx,\by) \big) \big| \leq c.
\end{equation*}
Subsequently, by integrating over the measure $P$ one has
\begin{equation*}
\big| \big( \int \lambda_{K_1,s}(\bx) \mathrm{d}P(\bx) \big) - \big( \int \lambda_{K_2,s}(\bx) \mathrm{d}P(\bx) \big) \big| \leq c.
\end{equation*}
\end{proof}

Equipped with such results, we now turn to the proof of Theorem \ref{thm:existence-dro-general-norm}.

\begin{proof}[Proof of Theorem \ref{thm:existence-dro-general-norm}]

The proof works in three steps. First, we show that we can reduce the problem to considering $\lambda = \lambda_{K,s}$ and $s \geq s_*(K)$ as defined in Lemma \ref{lem:uniform-bound}. In particular, note that for any feasible $(K,s,\lambda)$, we have that $\lambda \geq \lambda_{K,s}$ pointwise, so we may replace $\lambda$ by $\lambda_{K,s}$ with increasing the objective. Moreover, if $s < s_*(K)$, then $\lambda_{K,s}$ by Proposition \ref{prop:lambda-phi-characterization}, so the objective is $+\infty$. Therefore, every minimizing sequence can be taken to satisfy $$\lambda = \lambda_{K,s} \quad \text{and} \quad s \geq s_*(K).$$

Next, fix $s > 0$ and consider the restricted problem over the set $$\mathcal{K}_s:= \{K\ \text{star body} : \mathrm{vol}(K) \leq 1, s \geq s_*(K)\}.$$ For any minimizing sequence $\{K_n\} \subset \mathcal{K}_s$, define $g_n(\by) := \|\by\|_{K_n}$.  Denote $\mathbb{S}_{\|\cdot\|}:=\{\bu \in \R^d: \|\bu\|=1\}.$ By Lemma \ref{lem:dro-lipschitz}, we have that $g_n$ is $s$-Lipschitz on $\mathbb{S}_{\|\cdot\|}$. By Lemma \ref{lem:uniform-bound}, the space of functions $\{g_n\}$ is uniformly bounded on $\mathbb{S}_{\|\cdot\|}$ as well. Hence $\{g_n\}$ is equicontinuous and bounded on the compact metric space $(\mathbb{S}_{\|\cdot\|},\|\cdot\|)$. By Arzel\`{a}-Ascoli, there is a uniformly convergent subsequence $g_{n_k} \rightarrow g_{\infty}$ on $\mathbb{S}_{\|\cdot\|}.$ Extend this sequence and $g_{\infty}$ to $\R^d$ by $1$-homogeneity. Note that the extension of $g_{\infty}$ yields a gauge $\|\cdot\|_{K_{\infty}}$ of a star body $K_{\infty}$; moreover, $\mathrm{vol}(K_{\infty})\leq 1$ by lower semicontinuity of the volume functional. Finally, continuity of the objective with respect to $\|\cdot\|_K$ (via Lemma \ref{lem:cts_wrt_hausdorff}) and the choice $\lambda = \lambda_{K,s}$ imply that $K_{\infty}$ attains the minimum for this fixed $s$.

Finally, we minimize over $s$. In particular, let $\phi(s):=\min_{K \in \mathcal{K}_s} s \epsilon + \int \lambda_{K,s}\mathrm{d}P$. Because $\lambda_{K,s} \leq s\|\cdot\|$, we have that $\phi(s) \leq s\epsilon + s\mathbb{E}_P[\|\bx\|]$ (which is finite by assumption). On the other hand, $\phi(s) \geq s\epsilon$. Thus, any minimizing sequence $\{s_n\}$ is bounded (otherwise, the term $s\epsilon$ would drive the objective to $+\infty$). Extract a convergent subsequence $s_{n_k} \rightarrow s^* \geq 0$. For each $k$, pick a minimizer $K_{n_k} \in \mathcal{K}_{s_{n_k}}$. By the same compactness argument as above, along a further subsequence $K_{n_k} \rightarrow K^*$. Passing to the limit in the constraints yields $s^* \geq s_*(K^*)$; passing to the limit in the objective using Lemma \ref{lem:cts_wrt_hausdorff} and $\lambda = \lambda_{K,s}$ gives optimality of $(K^*,s^*,\lambda_{K^*,s^*})$.
\end{proof}

\section{Enforcing Convexity of the Optimal Regularizer} \label{sec:convex}

In this section we consider the problem of describing the optimal {\em convex} regularizer for a data source.  We formulate this problem as the following shape regression task.
\begin{equation} \label{eq:optimalconvexregularizer}
    \underset{K \in \mathcal{S}^{d}}{\mathrm{argmin}} ~~ \E_P[\|\bx\|_K] \qquad \mathrm{s.t.} \qquad \mathrm{vol} (K) = 1, K ~ \mathrm{convex}. 
\end{equation}
The basic questions we wish to investigate in this section are: (i) what is the underlying geometry of the optimal solutions to \eqref{eq:optimalconvexregularizer}, and (ii) what are the distributional robustness properties concerning the solutions to \eqref{eq:optimalconvexregularizer}.

Prior work in optimal regularization has characterized in some instances what the best convex regularizer would be \cite{leong2025optimal,traonmilin2024theory}. For example, the following Corollary of Theorem \ref{thm:optstarbodyreg} gave a condition on when the optimal star body regularizer is in fact convex. \begin{corollary}[Corollary 1 in \cite{leong2025optimal}]
    Let $P$ be a probability measure as in Theorem \ref{thm:optstarbodyreg}. Then if the function $\bx \mapsto 1/\rho_P(\bx)$ is a convex function on $\R^d$, then the optimal $\hat{K}$ is in fact a convex body and hence the optimal regularizer $\|\cdot\|_{\hat{K}}$ is convex.
\end{corollary}

It is unfortunately challenging to provide closed form expressions of the optimal solutions to \eqref{eq:optimalconvexregularizer} in a similar fashion as we did for star bodies.  The reason is because convexity introduces dependencies between the gauge function evaluations across neighboring points; in contrast, for star bodies, the gauge function evaluations between pairs of points were decoupled.  Our next best option is to seek finite-dimensional optimization problems that solve \eqref{eq:optimalconvexregularizer} approximately.  Wherever possible, we wish to pose these optimization instances as convex programs.  In what follows, we explain how this is possible in settings where data lies $\mathbb{R}^2$, and we describe how these ideas may be extended to higher dimensions.

\subsection{Parameterizing Convex Bodies}

A central challenge is to obtain a tractable parametrization of convex bodies. Two standard dual perspectives are available: representing $K$ as the convex hull of its extreme points, or as intersection of suporting half-spaces. In the following, we adopt the former perspective.  More concretely, suppose we parameterize $K$ as the convex hull of vectors of the form
\begin{equation*}
K = \mathrm{conv} \left( \{ \bu_i / t_i \}_{i=1}^{n} \right).
\end{equation*}
Here, $\bu_i \in \sd$ are unit vectors, while $t_i$ are positive scalars. The vectors $\bu_i$ are input variables specified beforehand and remained {\em fixed} throughout.  The scalar variables $t_i$ are the decision variables in our formulation.  

In essence, we want the variables $t_i$ to model the gauge function evaluation of $K$ in the direction $\bu_i$.  As such, it is necessary to impose conditions on the variables $t_i$ so that this conditions is indeed true.  In particular, one has
$$
\| \bu_i \|_{K} = \inf \{ t > 0 : \bu_i \in t \cdot K \} = \inf \{ t > 0 : \bu_i \in t \cdot \mathrm{conv} \left( \{ \bu_i / t_i \}_{i=1}^{n} \right) \} \leq t_i.
$$
The inequality follows from the fact that a blow-up of $\bu_i / t_i$ by a factor of $t_i$ equals $\bu_i$, and hence $\| \bu_i \|_{K}$ must be smaller.  It is also immediate to see that equality in the above holds if and only if $\bu_i / t_i$ lies on the boundary of $K$.  As such, our next objective is to describe conditions on $t_i$ that ensures all the vectors $\{\bu_i/t_i\}$ are extremal in $K$.

\subsection{Deriving Convexity Constraints in $\mathbb{R}^2$}  

Let's start simple by supposing data resides in $\mathbb{R}^2$.  Let $\{\bu_i\}_{i=1}^{n} \subset \mathbb{S}^{1}$ be a collection of direction vectors.  For concreteness, we suppose that these angles are denoted by $\theta_i$ in increasing order; that is
$$
\bu_i = (\cos(\theta_i), \sin (\theta_i) )^T,
$$
where the angles are chosen in order so that $ 0 \leq \theta_1 < \theta_2 < \ldots < \theta_n < 2\pi$.  We impose the condition that $\theta_{i} - \theta_{i-2} < \pi$.  This has the geometric interpretation that two consecutive direction vectors should not be too far apart.

Consider the following points in $\mathbb{R}^2$
\begin{equation*}
\bx_{-1} = \frac{1}{t_{-1}} \left( \begin{array}{c} \alpha_{-1}\\ \beta_{-1} \end{array} \right),
\bx_{0} = \frac{1}{t_{0}} \left( \begin{array}{c} \alpha_{0}\\ \beta_{0} \end{array} \right),
\bx_{1} = \frac{1}{t_{1}} \left( \begin{array}{c} \alpha_{1}\\ \beta_{1} \end{array} \right).
\end{equation*}
Here, we denote $\alpha_i = \cos(\theta_i)$ and $\beta_i = \sin(\theta_i)$.  Our next step is to derive conditions on $t_{-1},t_0,t_1$ that ensure $\bx_{0}$ point is extremal.  
Consider the triangles $\triangle_{-1,0} = \mathrm{conv}(\{ \mathbf{0}, \bx_{-1}, \bx_{0} \})$, $\triangle_{0,1} = \mathrm{conv}(\{ \mathbf{0}, \bx_{0}, \bx_{1} \})$, $\triangle_{-1,1} = \mathrm{conv}(\{ \mathbf{0}, \bx_{-1}, \bx_{1} \})$.  The condition that $\bx_0$ is extremal is equivalent to requiring that the area of triangle $\triangle_{-1,1}$ is smaller than the sum of the areas of $\triangle_{-1,0}$ and $\triangle_{0,1}$.  This yields
\begin{equation*}
( \alpha_{-1} \beta_1 - \alpha_1 \beta_{-1} ) / (t_{-1} t_0) + ( \beta_0 \alpha_{-1} - \beta_{-1} \alpha_0 ) / (t_{0} t_{1}) \geq ( \alpha_0 \beta_1 - \alpha_1 \beta_0 ) / (t_{-1} t_{1}).
\end{equation*}
Suppose we denote
\begin{equation*}
D_{i,j} = \alpha_i \beta_{j} - \alpha_{j} \beta_i.
\end{equation*}
This yields the inequality 
\begin{equation} \label{eq:convexityconstraint_R2}
D_{-1,0} t_0 \leq D_{0,1} t_1 + D_{-1,1} t_{-1}.
\end{equation}

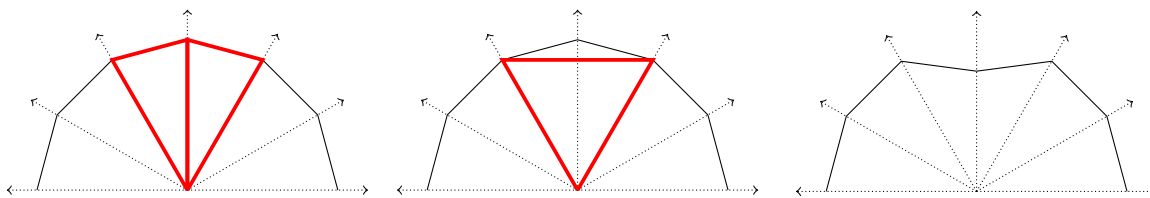
\begin{figure}[h]
\centering
\begin{subfigure}[t]{0.3\textwidth}
    \centering
\begin{tikzpicture}


\draw[densely dotted,->] (0,0) -- (1.2*2,1.2*0);
\draw[densely dotted,->] (0,0) -- (1.2*1.732,1.2*1);
\draw[densely dotted,->] (0,0) -- (1.2*1,1.2*1.732);
\draw[densely dotted,->] (0,0) -- (1.2*0,1.2*2);
\draw[densely dotted,->] (0,0) -- (-1.2*1,1.2*1.732);
\draw[densely dotted,->] (0,0) -- (-1.2*1.732,1.2*1);
\draw[densely dotted,->] (0,0) -- (-1.2*2,1.2*0);

\draw (2,0) -- (1.732,1) -- (1,1.732) -- (0,2) -- (-1,1.732) -- (-1.732,1) -- (-2,0);

\draw [line width=0.5mm, red ] (0,0) -- (-1,1.732) -- (0,2) -- (0,0);
\draw [line width=0.5mm, red ] (0,0) -- (1,1.732) -- (0,2) -- (0,0);

\end{tikzpicture}

\end{subfigure}%
~
\begin{subfigure}[t]{0.3\textwidth}
    \centering
\begin{tikzpicture}



\draw[densely dotted,->] (0,0) -- (1.2*2,1.2*0);
\draw[densely dotted,->] (0,0) -- (1.2*1.732,1.2*1);
\draw[densely dotted,->] (0,0) -- (1.2*1,1.2*1.732);
\draw[densely dotted,->] (0,0) -- (1.2*0,1.2*2);
\draw[densely dotted,->] (0,0) -- (-1.2*1,1.2*1.732);
\draw[densely dotted,->] (0,0) -- (-1.2*1.732,1.2*1);
\draw[densely dotted,->] (0,0) -- (-1.2*2,1.2*0);

\draw (2,0) -- (1.732,1) -- (1,1.732) -- (0,2) -- (-1,1.732) -- (-1.732,1) -- (-2,0);

\draw [line width=0.5mm, red ] (0,0) -- (-1,1.732) -- (1,1.732)  -- (0,0);

\end{tikzpicture}
\end{subfigure}
~
\begin{subfigure}[t]{0.3\textwidth}
    \centering
\begin{tikzpicture}



\draw[densely dotted,->] (0,0) -- (1.2*2,1.2*0);
\draw[densely dotted,->] (0,0) -- (1.2*1.732,1.2*1);
\draw[densely dotted,->] (0,0) -- (1.2*1,1.2*1.732);
\draw[densely dotted,->] (0,0) -- (1.2*0,1.2*2);
\draw[densely dotted,->] (0,0) -- (-1.2*1,1.2*1.732);
\draw[densely dotted,->] (0,0) -- (-1.2*1.732,1.2*1);
\draw[densely dotted,->] (0,0) -- (-1.2*2,1.2*0);

\draw (2,0) -- (1.732,1) -- (1,1.732) -- (0,1.6) -- (-1,1.732) -- (-1.732,1) -- (-2,0);

\end{tikzpicture}
\end{subfigure}

\caption{Illustration of how convexity for planar sets is enforced: The sum of the areas of the sectors in the red triangles in left sub-figure should exceed the area of the sector in the middle sub-figure.  Right sub-figure: When the inequality is violated, the resulting set is no longer convex.}
\label{fig:2dconvexity}
\end{figure}

\begin{figure}[h]
\centering
\begin{tikzpicture}
\draw[densely dotted,->] (0,0) -- (1.2*2,1.2*0);
\draw[densely dotted,->] (0,0) -- (1.2*1.732,1.2*1);
\draw[densely dotted,->] (0,0) -- (1.2*1,1.2*1.732);
\draw[densely dotted,->] (0,0) -- (1.2*0,1.2*2);
\draw[densely dotted,->] (0,0) -- (-1.2*1,1.2*1.732);
\draw[densely dotted,->] (0,0) -- (-1.2*1.732,1.2*1);
\draw[densely dotted,->] (0,0) -- (-1.2*2,1.2*0);
\draw[densely dotted,->] (0,0) -- (-1.2*1.732,-1.2*1);
\draw[densely dotted,->] (0,0) -- (-1.2*1,-1.2*1.732);
\draw[densely dotted,->] (0,0) -- (-1.2*0,-1.2*2);
\draw[densely dotted,->] (0,0) -- (1.2*1,-1.2*1.732);
\draw[densely dotted,->] (0,0) -- (1.2*1.732,-1.2*1);
\draw[densely dotted,->] (0,0) -- (1.2*2,-1.2*0);

\draw (0.9*2,0.9*0) -- (1.732,1) -- (1.1*1,1.1*1.732) -- (1.1*0,1.1*2) -- (-1,1.732) -- (-0.9*1.732,0.9*1) -- (-2,0) -- (-1.732,-1) -- (-1.1*1,-1.1*1.732) -- (0,-2) -- (1,-1.732) -- (1.732,-1) -- (0.9*2,0.9*0); 

\filldraw[color=red!60,opacity=0.5] (0,0) -- (1.1*0,1.1*2) -- (1.1*1,1.1*1.732) -- (0,0) -- cycle;

\end{tikzpicture}
\caption{Convex planar set expressed via a union of triangle sectors.  The volume of this set is expressed as the sum of the areas of each sector.}
\label{fig:2dsumofareas}
\end{figure}
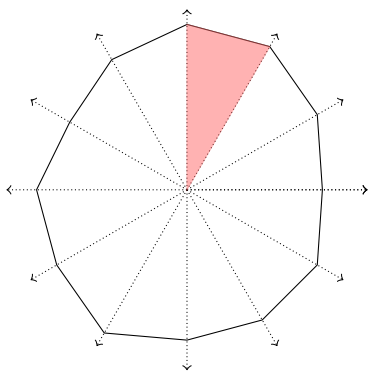

\subsection{Description of Convex Program} 

Next, we assume that the probability distribution $P$ of interest is supported only on the points $\{ \bu_i \} \subset \R^2$.  In what follows, we denote
$$
a_i := \mathbb{P} [\bx = \bu_i].
$$
We also let $\theta_{i,i+1}$ denote the angle between the directions $\bu_i$ and $\bu_{i+1}$. 

The optimization instance \eqref{eq:optimalconvexregularizer} can be expressed via the following convex program
\begin{equation} \label{eq:convexprogram_r2}
\begin{aligned}
\min ~~ & \sum a_i t_i \\
\mathrm{s.t.} ~~ & t_i D_{i-1,i+1} \leq t_{i-1} D_{i,i+1} + t_1 D_{i-1,i} \\
& \sum_{i=1}^{n} (1/2) \sin \theta_{i,i+1} /( t_i t_{i+1}) \leq 1 .
\end{aligned}
\end{equation}
We explain how one arrives at \eqref{eq:convexprogram_r2}.

First, as discussed earlier, the inequality $t_i D_{i-1,i+1} \leq t_{i-1} D_{i,i+1} + t_1 D_{i-1,i}$ ensures that the points $\bu_i/t_i$ are extreme points of $K$.  By doing so, we ensure that the gauge function evaluation with respect to $K$ in the direction $\bu_i$ is exactly $t_i$.  

Second, the objective can be expressed as follows
$$
\mathbb{E}_{P} [ \| \bx \|_K ] = \sum_{i=1}^{n} \mathbb{P} [\bx=\bu_i] \| \bu_i \|_K = \sum_{i=1}^{n} a_i t_i \| \bu_i \|_2 = \sum a_i t_i.
$$
Here, the second equality relies on the fact that the gauge function of $K$ in the direction $\bu_i$ is exactly $t_i$.  

Third, the inequality $\sum_{i=1}^{n} (1/2) \sin \theta_{i,i+1} /( t_i t_{i+1}) \leq 1$ models the constraint that the volume of $K$ is at most one.  As a reminder, the area of the sector spanned by $\bu_i/t_i$ and $\bu_{i+1}/t_{i+1}$ is $(1/2) \sin \theta_{i,i+1} /( t_i t_{i+1})$, and the area of $K$ is the sum of the area of each sector -- see Figure \ref{fig:2dsumofareas} for an illustration.

For concreteness, suppose we take $\bu_k = (\cos (2 \pi k/n), \sin ( 2\pi k/n))^T$.  Then \eqref{eq:convexprogram_r2} simplifies to
\begin{equation} \label{eq:convexprogram_r3}
\begin{aligned}
\min ~~ & \sum a_i t_i \\
\mathrm{s.t.} ~~ & t_i \sin(4\pi/n) \leq (t_{i-1} + t_{i+1}) \sin(2\pi/n) \\
& \sum_{i=1}^{n}1 /( t_i t_{i+1}) \leq 2 / \sin(2\pi/n).
\end{aligned}
\end{equation}

We briefly justify why \eqref{eq:convexprogram_r2} specifies a convex program.  The objective and the first set of constraints are linear.  The Hessian of the function $f(x_1,x_2) = 1/(x_1 x_2)$ is
\begin{equation*}
\nabla^2 f = \frac{1}{x_1 x_2 } \left( \begin{array}{cc}
2/x_1^2 & 1/(x_1 x_2) \\ 1/(x_1 x_2) & 2/x_2^2
\end{array} \right).
\end{equation*}
The determinant is $3/(x_1^2 x_2^2)$ which is positive over $x_1>0,x_2>0$, and hence $f$ is convex over the non-negative orthant.

\subsection{Numerical Illustrations}

\begin{figure}[h] \label{fig:dro_L0L1}
\centering
\includegraphics[width=0.24\textwidth]{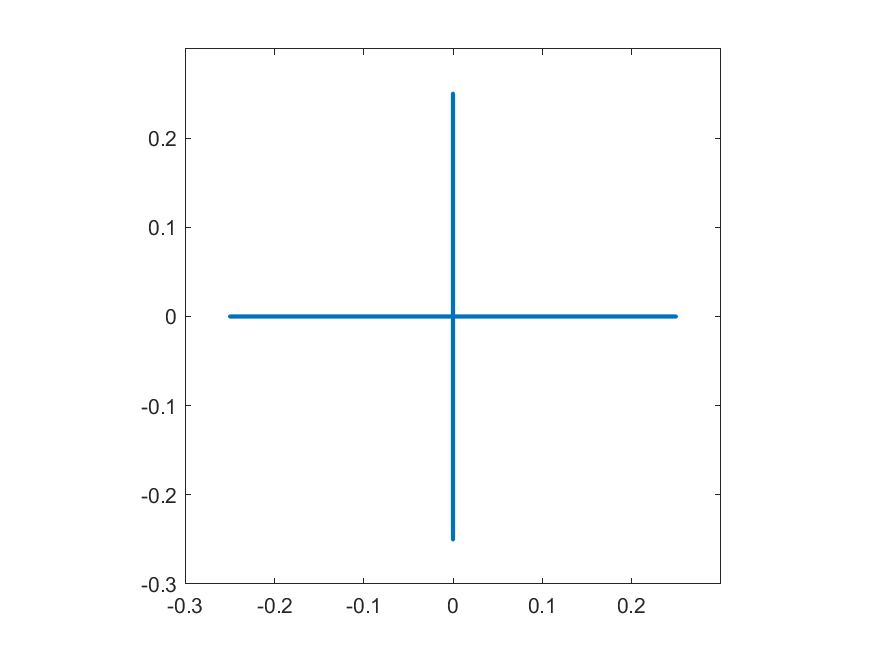}
\includegraphics[width=0.24\textwidth]{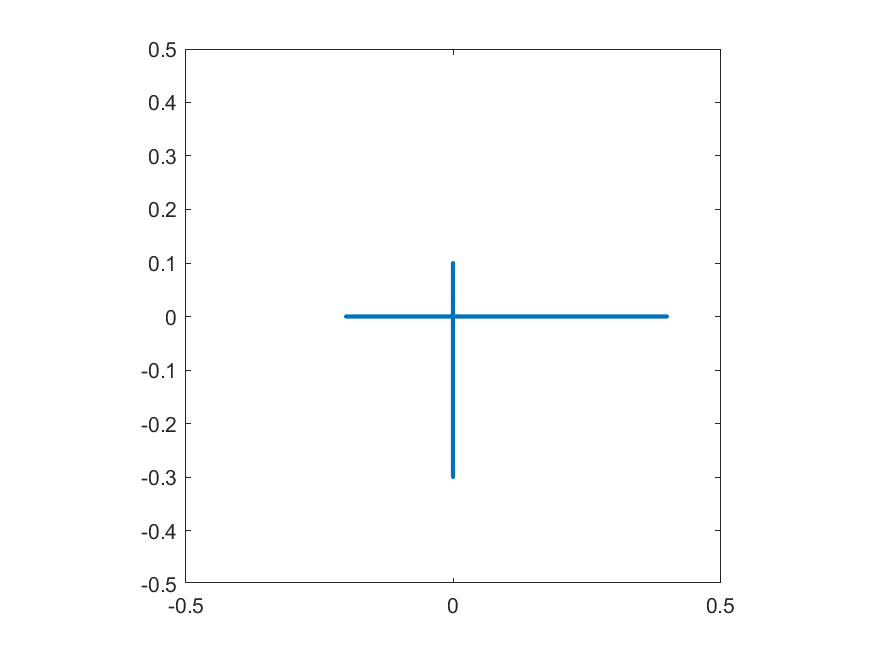}
\includegraphics[width=0.24\textwidth]{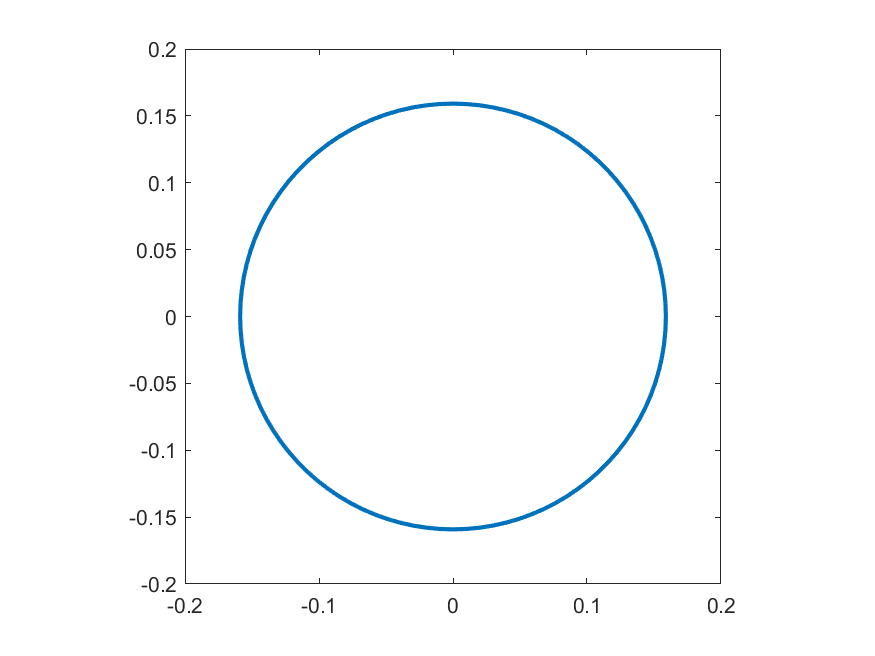}
\includegraphics[width=0.24\textwidth]{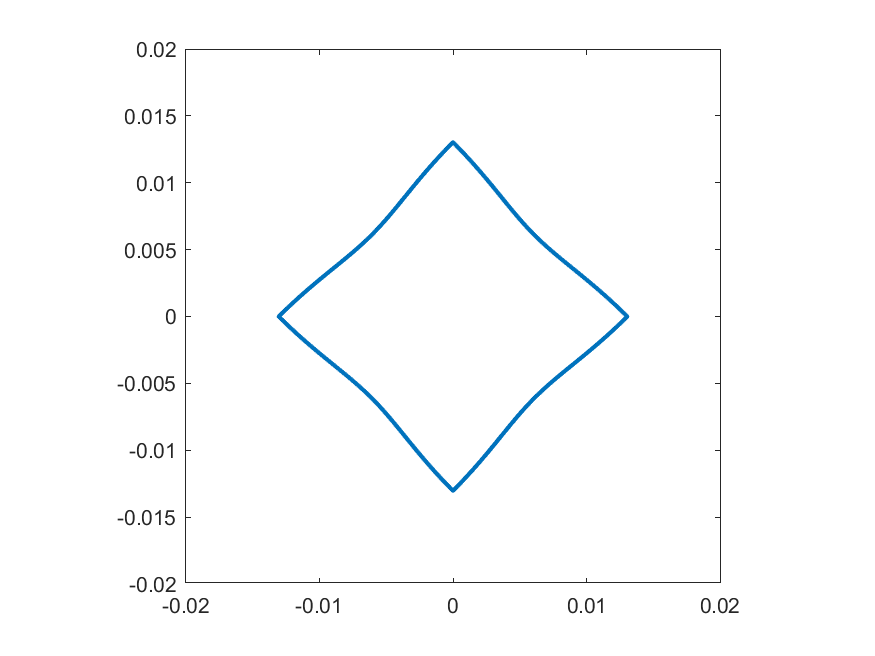}
\includegraphics[width=0.24\textwidth]{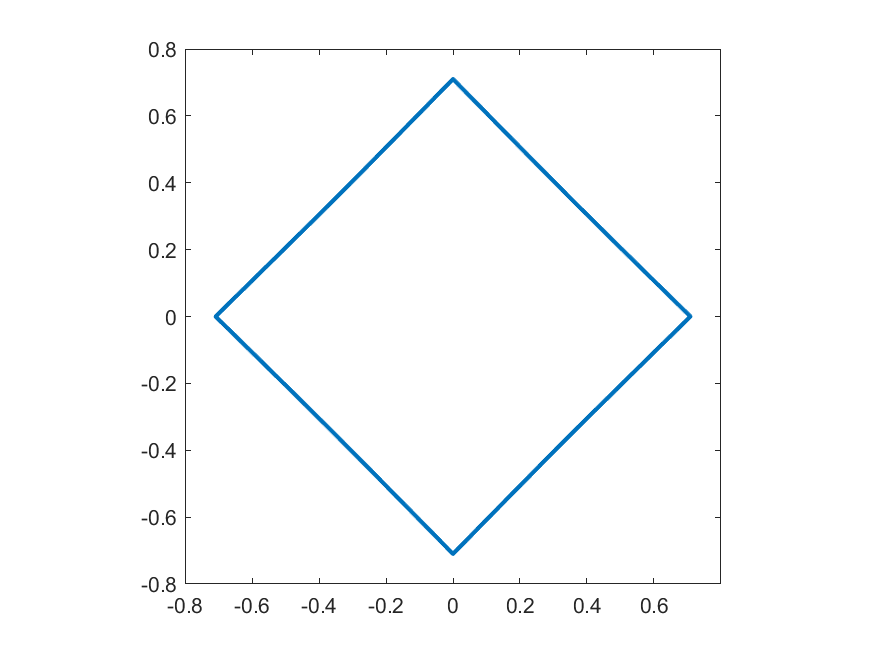}
\includegraphics[width=0.24\textwidth]{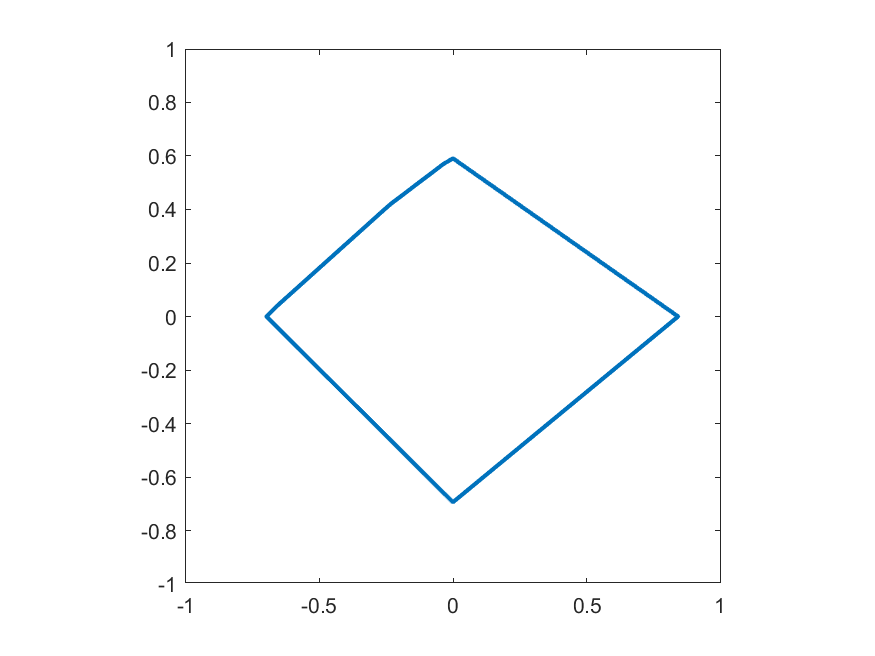}
\includegraphics[width=0.24\textwidth]{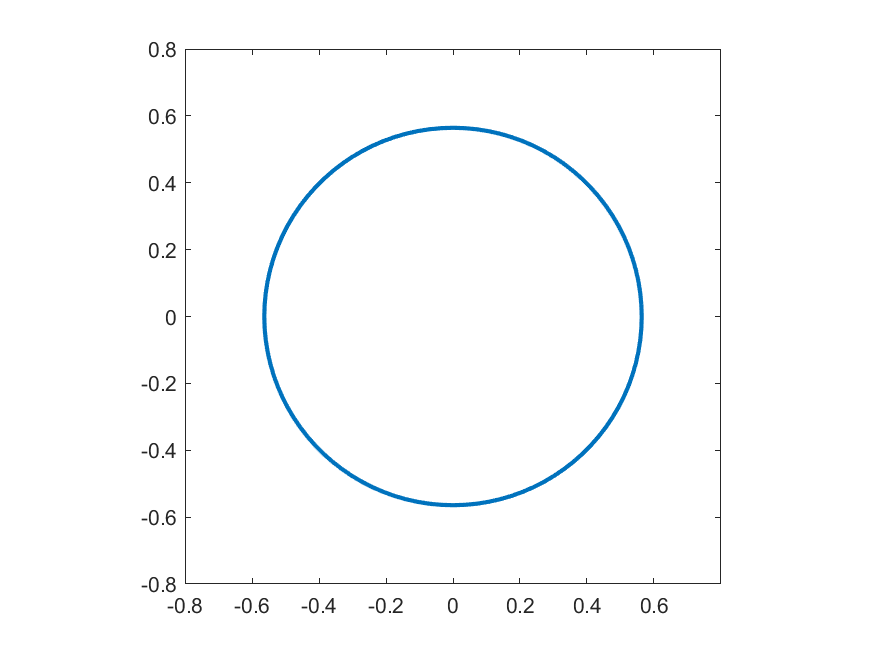}
\includegraphics[width=0.24\textwidth]{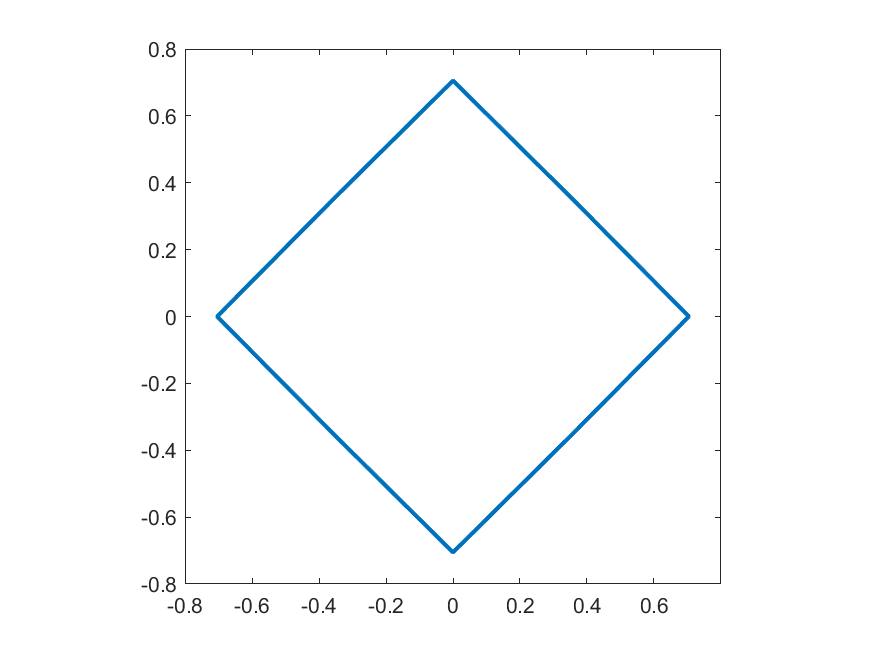}
\caption{Optimal Convex Regularizers.  The underlying data distribution is given in the top row and the (level set of the) corresponding optimal convex regularizer is specified in the bottom row.}
\end{figure}

We consider computing the optimal convex regularizer for a number of different distributions using the convex program we presented in the earlier section. These examples are computing using \eqref{eq:convexprogram_r2}.

\paragraph{Example 1: Uniform basis vectors.} In the first example we consider data distributed uniformly on the standard basis vectors $\{ (0,1), (-1,0), (0,-1), (1,0) \}$.  As we expect, the optimal regularizer in this case is indeed the L1-ball.  Note that the optimal non-convex regularizer does not exist without suitable regularity assumptions put in place.  This is because the distribution is atomic and the optimal star "body" would have zero volume.

\paragraph{Example 2: Weighted basis vectors.} In the second example, we consider data distributed on the same set of vectors, but with distribution $0.1,0.2,0.3,0.4$ respectively.  In this case, the optimal regularizer is a different polytope, whose gauge function evaluations at the standard basis vectors are weighted differently as a response to the data observed.

\paragraph{Example 3: Uniform on the circle.} In the third example, data is distributed uniformly over the unit-circle.  As we expect, the optimal regularizer is the L2-norm.

\paragraph{Example 4: Laplace distribution.} In the fourth example, the distribution is a Laplace-type distribution whose distribution is $\exp(- \| \bx \|_1)$, and subsequently normalized to be a distribution.  The optimal regularizer is the L1-norm, and this confirms an observation made in \cite{leong2025optimal} regarding distributions that are functions of level sets of the regularizer.

\begin{figure}[h] 
\centering
\includegraphics[width=0.24\textwidth]{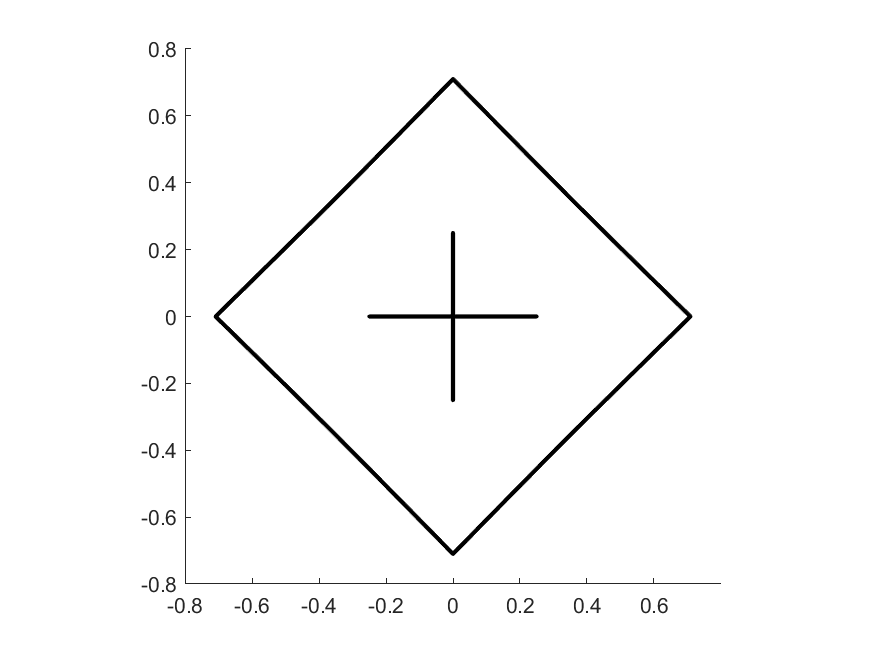}
\includegraphics[width=0.24\textwidth]{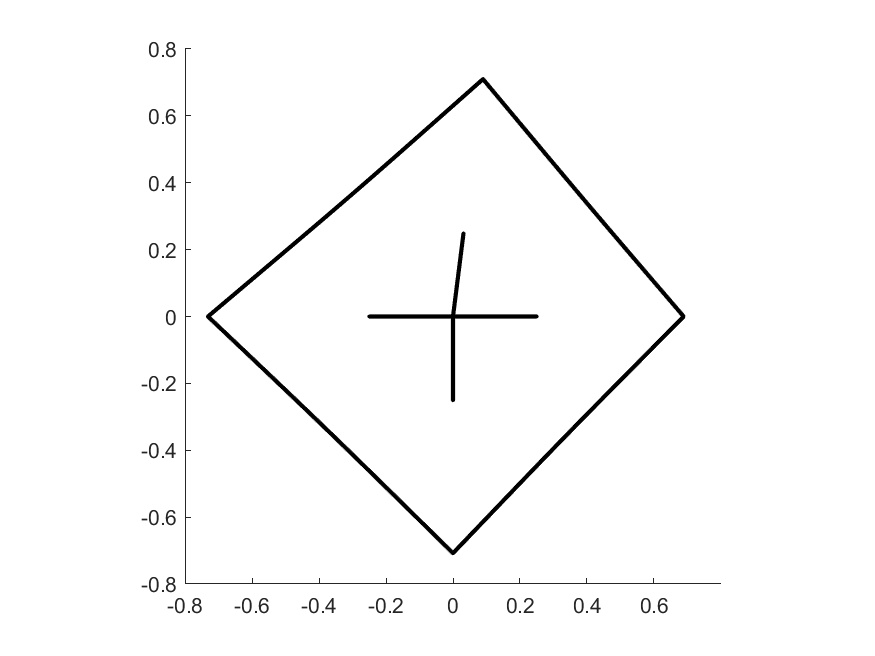}
\caption{Small changes in the underlying distribution lead to small changes in the optimal convex regularizer.  The underlying data distribution are atomic measures.  The two distributions differ by a small shift in the support set.}
\label{fig:dro_cvx}
\end{figure}

{\bf Distributional shifts.}  Our next question is, how does the optimal convex regularizer change with respect to changes in the distribution?  We do not have a complete understanding of this question, but believe the answer is that optimal convex regularizers tend to be well-behaved with respect to small changes in the underlying distribution. To illustrate this point, we consider the following stylistic set-up.  In the first example, the underlying distribution is uniform over $\{ (0,1), (-1,0), (0,-1), (1,0) \}$.  In the second example, the underlying distribution is uniform over $\{ (\sin(2\pi/25),\cos(2\pi/25)), (-1,0), (0,-1), (1,0) \}$.  In particular, there is a small shift in the first data-point.  In Figure \ref{fig:dro_cvx} we show the corresponding optimal convex regularizers -- these are superimposed over the underlying distributions.

\subsubsection{Robustness induced by convexity}

Based on the previous example, we notice that optimal convex regularizers for distributions $P_1$ and $P_2$ that are close to one another are also similar in the sense that the gauge function evaluations differ by a small amount, uniformly across all unit-norm inputs, i.e., 
\begin{equation*}
\left\| \|\cdot\|_{\hat{K}(P_1)} - \|\cdot\|_{\hat{K}(P_2)} \right\|_{\infty}\ \text{is small}.
\end{equation*}
This is in sharp contrast with optimal non-convex regularizers where small shifts can lead to large changes in the gauge function evaluations, specifically at the locations where the ``support" changes.

While the example we describe is stylistic, it offers some basic intuition. Specifically, the convexity shape constraint ensures that the resulting regularizer is well-behaved for any input, leading to a natural type of robustness.  As a consequence, even with small changes in the underlying distribution, the gauge function of the corresponding optimal regularizer changes smoothly. To provide theoretical support for this, we give the following result which shows that the optimal convex gauge should also behave well for nearby distributions. This explicitly depends on the distance between the two distributions along with their Lipschitz constants.
\begin{proposition} \label{prop:convex-robustness}
 For any two probability measures $P,Q$ and minimizers $\hat{K}_P$ and $\hat{K}_Q$ to \eqref{eq:optimalconvexregularizer}, we have that \begin{align*}
    \mathbb{E}_Q[\|\bx\|_{\hat{K}_P}] \leq \inf_{K\ \mathrm{convex},\  \mathrm{vol}(K)=1}\mathbb{E}_Q[\|\bx\|_K] + \left(\mathrm{Lip}(\hat{K}_P)+ \mathrm{Lip}(\hat{K}_Q)\right)W_1(P,Q).
    \end{align*}
\end{proposition}
\begin{proof}
    Note that since $\hat{K}_P$ and $\hat{K}_Q$ are convex bodies with the origin in their interiors, their Lipschitz constants are finite. Moreover, by Kantorovich-Rubinstein duality \cite{villani2021topics}, we have that for any $L$-Lipschitz function and probability measures $P,Q$, $$|\mathbb{E}_Q[f] - \mathbb{E}_P[f]| \leq LW_1(P,Q).$$ Using the Lipschitz bound, we have that \begin{align*}
        \mathbb{E}_Q[\|\bx\|_{\hat{K}_P}] & = \mathbb{E}_Q[\|\bx\|_{\hat{K}_P}] - \mathbb{E}_P[\|\bx\|_{\hat{K}_P}] + \mathbb{E}_P[\|\bx\|_{\hat{K}_P}]  \\
        & \leq \mathbb{E}_P[\|\bx\|_{\hat{K}_P}]  + \mathrm{Lip}(\hat{K}_P)W_1(P,Q) \\
        & \leq \mathbb{E}_P[\|\bx\|_{\hat{K}_Q}]  + \mathrm{Lip}(\hat{K}_P)W_1(P,Q)
    \end{align*} where in the last line we used optimality of $\hat{K}_P.$ Now, considering $\hat{K}_Q$, we have that \begin{align*}
        \mathbb{E}_P[\|\bx\|_{\hat{K}_Q}] & = \mathbb{E}_P[\|\bx\|_{\hat{K}_Q}] -\mathbb{E}_Q[\|\bx\|_{\hat{K}_Q}] + \mathbb{E}_Q[\|\bx\|_{\hat{K}_Q}] \\
        & \leq \inf_{K\ \text{convex},\  \mathrm{vol}(K)=1} \mathbb{E}_Q[\|\bx\|_K] + L(\hat{K}_Q)W_1(P,Q) 
    \end{align*} where we used the definition of $\hat{K}_Q$ in the last inequality. Combining the two inequalities, we see that \begin{align*}
\mathbb{E}_Q[\|\bx\|_{\hat{K}_P}] & \leq   \inf_{K\ \text{convex},\  \mathrm{vol}(K)=1} \mathbb{E}_Q[\|\bx\|_K] + \left(\mathrm{Lip}(\hat{K}_P)+ \mathrm{Lip}(\hat{K}_Q)\right)W_1(P,Q)    
    \end{align*} as desired.
\end{proof} \begin{remark}
This result shows how the \textit{performance} of an optimal convex regularizer for a different distribution $Q$ is controlled by the distance of the distribution to $P$. This also depends multiplicatively on the Lipschitz constants of the optimal convex regularizers for $P$ and $Q$. Note that for convex bodies, their Lipschitz constants are controlled by the size of the largest Euclidean ball contained in its interior. For nonconvex star bodies, it is possible for star body regularizers to have arbitrarily large Lipschitz constants, even if they contain a Euclidean ball in their interior, making a bound of the above form less useful.    
\end{remark}


%

\section{Alternative Proof Techniques and Extensions}
\label{sec:extensions}
While Theorem \ref{thm:optstarbodyreg} provides a characterization of optimal star-body regularizers via dual Brunn–Minkowski theory, it is also possible to arrive at the result through more elementary means. In this section we first present a constructive discretization-based proof, which reduces the infinite-dimensional optimization problem to a sequence of finite-dimensional convex programs. This perspective provides additional intuition: optimal gauges can be seen as limits of piecewise-constant approximations. 

Beyond this expository goal, we also show how the same discretization framework naturally accommodates extensions of the optimal regularization problem, including critic-based formulations and their distributionally robust variants. These examples highlight the flexibility of our approach and its connection to adversarially motivated learning paradigms.

\subsection{Elementary Discretization Proof of Theorem \ref{thm:optstarbodyreg}} \label{sec:piecewisestar}

In what follows, we provide an alternative proof of Theorem \ref{thm:optstarbodyreg}, one that is more elementary compared to the analysis in \cite{leong2025optimal}.  The basic idea is to write down a suitable discretization of \eqref{eq:opt_star_formulation} that leads to a finite dimensional convex program for which we are readily able to characterize the optimal regularizer. As in Section \ref{sec:dro-reformulation-convex-duality}, let $\mathcal{U} = \{ U : U \subset \sd \}$ be a collection of open subsets of the unit sphere $\sd$ that form a partition of the sphere $\sd$, up to a set of zero measure.  Consider the collection of star sets $K$ whose radial functions are piecewise {\em constant} over each $U \in \mathcal{U}$.  With a slight abuse of notation, we simply say that $K$ is piecewise constant over $\mathcal{U}$. We first obtain the optimal solution to the discretized version of \eqref{eq:opt_star_formulation}.

\begin{proposition} \label{thm:optstarreg_piecewiseconstant}
Let $P$ be a distribution on $\R^d$ with density $p$. Suppose $\rho_P$ is positive.  Consider 
\begin{align} \label{eq:optformulation_piecewiseconstant}
\underset{K \in \mathcal{S}^d}{\mathrm{argmin}}  ~~  \E_P[\|\bx\|_K] \quad \mathrm{s.t.} \quad \mathrm{vol}(K) = 1, K \text{ is piecewise constant over } \mathcal{U}.
\end{align}
Then the optimal solution is the star set $\hat{K}$ whose radial function over $\mathcal{U}$ is given by
\begin{equation} \label{eq:optstar_piecewiseconstant}
\rho(U) = \frac{c}{w(U)}\left( \int_{v \in U \cap \sd} \int_{r=0}^{\infty} r^d p(r\bv) \mathrm{d}\bv \right)^{1/(d+1)}.
\end{equation}
Here, $w(U)$ is the surface area of $U \in \sd$, and $c > 0$ is a scalar chosen such that $\mathrm{vol}(\hat{K}) = 1$.
\end{proposition}

\begin{proof}[Proof of Proposition \ref{thm:optstarreg_piecewiseconstant}]
Let $K$ denote a piecewise constant star body over $\mathcal{U}$. We let $t_U$ denote the value of the gauge function.  Then the volume of $K$ is given by 
\begin{equation*}
\propto ~~ \sum_{U \in \mathcal{U}} w_U / t_U^{d}.
\end{equation*}
where $w_U$ is the surface area of $U$.

Next, we evaluate the objective with respect to $K$.  Fix $U \in \mathcal{U}$.  Then, by integrating over spherical coordinates, one has
\begin{equation*}
\begin{aligned}
\mathbb{E} [ \|\bx\|_{K} \cdot 1\{\bx \in U \}] ~=~ & \int \|\bx\|_K p(\bx) \cdot 1\{ \bx \in U \} \mathrm{d}\bx \\
~=~ & \int_{\sd} \int_{r=0}^{\infty} r^d \|\bv\|_K p(r \bv) \cdot 1\{\bv \in U \} \mathrm{d}\bv \\
~=~ & t_U \int_{\bv \in U \cap \sd} \int_{r=0}^{\infty} r^d p(r \bv) \mathrm{d}\bv.
\end{aligned}
\end{equation*}

Denote $\alpha_U = \int_{\bv \in U} \int_{r=0}^{\infty} r^d p(r \bv) \mathrm{d}\bv$.  Then the optimization instance \eqref{eq:opt_star_formulation} where we restrict to star sets that are piecewise constant over $\mathcal{U}$ is given by
\begin{equation} \label{eq:starbody_finitedimopt}
\underset{t_U > 0}{\arg \min} ~~ \sum_{U \in \mathcal{U}} \alpha_U t_U \qquad \mathrm{s.t.} \qquad \sum_{U \in \mathcal{U}} w_U / t_U^d \leq 1.
\end{equation}

Note that this is a convex program.  In particular, the function $x \mapsto x^{-d}$ is convex over the domain $x>0$. We proceed to solve the optimization instance.  The Lagrangian is
\begin{equation*}
\mathcal{L} := \sum \alpha_U t_U + \lambda (\sum w_U / t_U^d -1).
\end{equation*}
The derivative of $\mathcal{L}$ with respect to $t_U$ is
\begin{equation*}
\frac{ d \mathcal{L} }{d t_U } = \alpha_U - \lambda d w_U / t_U^{d+1}.
\end{equation*}
Therefore, at optimality, one has 
\begin{equation*}
t_U = \Big( \frac{\alpha_U}{\lambda d w_U} \Big)^{1/(d+1)}.
\end{equation*}

The solution coincides with \eqref{eq:optstar_piecewiseconstant}, up to an unknown parameter $\lambda$.  The objective $\alpha_U t_U$ decreases monotonically as $\lambda$ increases.  So $\lambda$ is chosen as large as possible such that the constraint $\sum_{u \in \mathcal{U}} w_U / t_U^d \leq 1$ is satisfied with equality.
\end{proof}

It is perhaps clear that Theorem \ref{thm:optstarbodyreg} should be the continuous limit of Proposition \ref{thm:optstarreg_piecewiseconstant}.  The objective of this section is to formalize the process of taking limits as the discretization goes to zero.  This subsection may be skipped without affecting the readability of the remaining sections. To make these steps precise, we will construct star sets that are piecewise constant over partitions that become increasing smaller.  Formally:

\begin{definition}[Refining partitions]
We say that a sequence of partitions (of the unit sphere) $\{ \mathcal{U}_t \}_{t=1}^{\infty}$ is a sequence of {\em refining partitions} if there exists $\kappa < 1$ such that
\begin{enumerate}
\item $\mathrm{diam}(U) \leq \kappa^t $ for all $ U \in \mathcal{U}_t$, and
\item if $U_s \in \mathcal{U}_s$ and $U_t \in \mathcal{U}_t$ where $s \leq t$, then either $U_t \subset U_s$ or $\mu(U_t \cap U_s) = 0$.
\end{enumerate}
\end{definition}


Next, we prove a series of helpful technical results.  Consider a sequence of refining partitions $\{ \mathcal{U}_t \}_{t=1}^{\infty}$.  Let $\hat{K}(\mathcal{U}_t)$ be the optimal star set according to Proposition \ref{thm:optstarreg_piecewiseconstant}.  Our next result shows that $\hat{K}(\mathcal{U}_t) \rightarrow \hat{K}$, the optimal solution in Theorem \ref{thm:optstarbodyreg}. 

\begin{proposition} \label{thm:convergence_of_optimalapproxs}
Suppose $\rho_P$ is positive and continuous.  Then
\begin{equation*}
\rho_{\hat{K}(\mathcal{U}_t)} (\bu) \rightarrow \rho_{\hat{K}} (\bu) \quad \text{for all} \quad \bu \in \mathbb{S}^{d-1}.
\end{equation*}
In particular, because $\rho_{\hat{K}}$ is continuous over a compact domain, the convergence is also uniform.  In particular, this implies 
\begin{equation*}
\| \rho_{\hat{K}(\mathcal{U}_t)} - \rho \|_{\infty} \rightarrow 0 \qquad \text{as}
 \qquad t \rightarrow \infty,
\end{equation*}
which implies
\begin{equation*}
\hat{K}(\mathcal{U}_t) \rightarrow \hat{K} \qquad \text{ in the radial Hausdorff metric}.
\end{equation*}
\end{proposition}

\begin{proof} [Proof of Proposition \ref{thm:convergence_of_optimalapproxs}]
Pick $\bu \in \mathbb{S}^{d-1}$.  Consider a sequence of sets $\{U_t\}_{t=1}^{\infty}$ where $U_t \in \mathcal{U}_t$, and $\bu \in U_t$; i.e., it is the set within each partition that contains $\bu$.  Set $\epsilon > 0$.  Because $\rho_p$ is continuous and the unit sphere is compact, the total variation of $\rho_P$ can be made arbitrary small for partitions $\mathcal{U}_t$ that are chosen sufficiently fine.  In particular, there is a $\delta$ such that for all sets $\overline{u} \in \mathcal{U}$ are such that $\mathrm{diam}(\overline{U}) \leq \delta$ then the total variation of $\rho_P$ is at $\epsilon$.  For such a partition, one has
\begin{equation*}
\Big | \Big( \frac{1}{w(\overline{U})}\int_{\bv \in \overline{U} \cap \mathbb{S}^{d-1}} \int_{r=0}^{\infty} r^d p(r\bv) \mathrm{d}\bv \Big)^{1/(d+1)} - \Big( \int_{r=0}^{\infty} r^d p(r\bu) dr \Big)^{1/(d+1)} \Big| \leq \epsilon.
\end{equation*}
In particular, by taking $\delta \rightarrow 0$, we conclude that $\rho \rightarrow \rho_P$ uniformly.  The other assertions in the result follow accordingly.
\end{proof}

Our next result is a technical lemma that shows:  Suppose $K$ is non-degenerate star.  Given any refining partition, it is possible to approximate $K$ as being piecewise constant over a sufficiently fine partition, up to any desired accuracy.

\begin{lemma} \label{thm:approxstarbody_piecewise}
Let $K$ be a star body, and suppose that $\epsilon_0 \cdot B \subset K$ for some $\epsilon_0 >0$.  Let $\{ \mathcal{U}_t \}_{t=1}^{\infty}$ be a sequence of refining partitions of $\mathbb{S}^{d-1}$.  Then for every $\epsilon > 0$, there exists a $t: = t(\epsilon)$ (depending on $\epsilon$) as well as a corresponding sequence of star sets $\{ K(\mathcal{U}_s) \}_{s=t}^{\infty}$ such that (i) $K(\mathcal{U}_s)$ is piecewise constant over $\mathcal{U}_s$, (ii) the radial function of $K(\mathcal{U}_s)$ satisfies $\| \rho_K - \rho_{K(\mathcal{U}_s)} \|_{\infty} \leq \epsilon$, and (iii) the volume of $K(\mathcal{U}_s)$ is one.
\end{lemma}

\begin{proof}
For the first part of the proof, we show that there exists a $t_1 := t_1(\epsilon)$ and a sequence $\{ \tilde{K} (\mathcal{U}_s) \}_{s=t}^{\infty}$ satisfying requirements (i) and (ii).  In the second part of the proof, we scale the sets $\tilde{K} (\mathcal{U}_s)$ to have unit volume.  We show that when the indices $t$ are sufficiently large, the conditions (i) and (ii) are still satisfied.  

In what follows, for any partition $\mathcal{U}$, we define $\tilde{K}(\mathcal{U})$ to be a star set that is piecewise constant over $\mathcal{U}$ whereby the value over $U \in \mathcal{U}$ is equal to the average of $\rho$ over $U$: $\frac{1}{\mu (U)} \int_{\bx \in U} \rho(\bx) \mathrm{d}\bx$.  

[(i) and (ii)]:  Since $K$ is a star body, the radial function $\rho:= \rho_K$ is continuous.  Since $\rho$ is continuous over a compact set, it is also uniformly continuous.   In particular, for every $\epsilon > 0$ is a $\delta > 0$ such that $|\rho(\bx) - \rho(\by)| \leq \epsilon $ for all pairs of $\bx,\by$ such that $\|\bx-\by\|_2 \leq \delta$.  In particular, if we set $t_1$ to be such that $\kappa^{t_1} \leq \delta$, then one has $|\rho(\bx) - \rho(\by)| \leq \epsilon$ for all $\bx,\by \in u$ where $u \in \mathcal{U}_t$.  This is because $\mathrm{diam}(U) \leq \kappa^{t_1} \leq \delta$.  Consequently, this implies $\| \rho_K - \rho_{\tilde{K}(\mathcal{U}_t)} \|_{\infty} \leq \epsilon$.  In what follows, given $\epsilon > 0$, we let $t_1(\epsilon)$ be the value of $t_1$ satisfying the above consequences.

[(iii)]:  Next, we define the sequence $\{ K(\mathcal{U}_t) \}_{t=1}^{\infty}$ where $K(\mathcal{U}_t) = \tilde{K}(\mathcal{U}_t) / \mathrm{vol}(\tilde{K}(\mathcal{U}_t))$.  Then the resulting sequence of sets have volume one.  We show that conditions (i) and (ii) remain satisfied provided $t$ is sufficiently large.

Because $\rho_K$ is continuous over a compact set $\mathbb{S}^{d-1}$, and because $\epsilon_0 \subset K$, one can bound $R_0 \leq \rho_K \leq R_1$ for some $R_0,R_1>0$.  Consider a partition $\mathcal{U}$, and define $\delta_2 := \| \rho_{K(\mathcal{U}_t)} - \rho_{K} \|_{\infty}$.  Because $R_1 \leq \rho_{K} \leq R_2$, one can bound $\mathrm{vol}(K) - \mathrm{vol}(K(\mathcal{U})) \leq \mathrm{vol}( R_1 \cdot B) - \mathrm{vol}( (R_1 - \delta_2) \cdot B) = (R_1^d - (R_1 - \delta_2)^d) \mathrm{vol}(B)$, where $B$ is the unit $\ell_2$-ball.  One then has
\begin{equation*}
\frac{|\mathrm{vol}(K)-\mathrm{vol}(K(\mathcal{U}))|}{\mathrm{vol}(K(\mathcal{U}))} \leq \frac{(R_1^d - (R_1 - \delta_2)^d) \cdot \mathrm{vol}(B)}{R_0^d \cdot \mathrm{vol}(B)} \leq c \delta_2,
\end{equation*}
for some constant $c$ independent of $\delta_2$, though possibly depending on $d$, $R_0$, and $R_1$.

Now set $t := \max \{ t_1(\epsilon / (2 c R_1)), t_1 (\epsilon/2)\}$.  Based on the earlier construction, we have a sequence $\{\mathcal{U}_s\}_{s=t}^{\infty}$ that satisfies $\|\rho_K - \rho_{K(\mathcal{U}_s)}\|_{\infty} \leq \epsilon / (2 c R_1)$ for all $s \geq t$ ($\epsilon / (2 c R_1)$ is $\delta_2$).  We then have
\begin{equation*}
\begin{aligned}
\| \rho_{\tilde{K}(\mathcal{U}_s)} - \rho_K \|_{\infty} = \, & \| (1/\mathrm{vol}(K(\mathcal{U}_s))) \rho_{K(\mathcal{U}_s)} - \rho_K \|_{\infty} \\
\leq \, & \| (1/\mathrm{vol}(K(\mathcal{U}_s))) \rho_{K(\mathcal{U}_s)} - \rho_{K(\mathcal{U}_s)} \|_{\infty} + \| \rho_{K(\mathcal{U}_s)} - \rho_K \|_{\infty} \\
= \, & |1 - 1/\mathrm{vol}(K(\mathcal{U}_s))| \| \rho_{K(\mathcal{U}_s)} \| + \| \rho_{K(\mathcal{U}_s)} - \rho_K \|_{\infty} \\
\leq \, & c \times (\epsilon/(2cR_1)) \times R_1 + \epsilon/2 = \epsilon.
\end{aligned}
\end{equation*}
In the last inequality, we apply the upper bound $\| \rho_{K(\mathcal{U}_s)} - \rho_K \|_{\infty} \leq \epsilon / 2$ using the fact that $t$ was chosen so that $t \geq t_1 (\epsilon/2)$.  We thus see that the constructed sequence satisfies all three conditions, as required.
\end{proof}

\begin{lemma} \label{thm:gaugeradial}
Suppose $K_1$ and $K_2$ are star sets.  Suppose it is known that $\epsilon  B^d \subset K_1$, and $\epsilon B^d \subset K_2$.  Then $\| \|\cdot\|_{K_1} - \|\cdot\|_{K_2} \|_{\infty} \leq (1/\epsilon^2) \| \rho_{K_1} - \rho_{K_2} \|_{\infty}$.
\end{lemma}

\begin{proof}
Note that since $\epsilon \cdot B \subset K_i$, we have $\rho_{K_i}(\bu) \geq \rho_{\epsilon \cdot B}(\bu)= \epsilon$ for any $\bu \in \sd$. This gives the following: 

\begin{align*} 
\| \|\cdot\|_{K_1} - \|\cdot\|_{K_2} \|_{\infty} = \| 1/\rho_{K_1} - 1/\rho_{K_2} \|_{\infty} & = \max_\bu | 1/\rho_{K_1}(\bu) - 1/\rho_{K_2}(\bu) | \\
& =  \max_\bu |\rho_{K_1}(\bu) - \rho_{K_2}(\bu) |/ (|\rho_{K_1}(\bu) \rho_{K_2}(\bu)|) \\
& \leq \max_\bu |\rho_{K_1}(\bu) - \rho_{K_2}(\bu) |/ \epsilon^2 \\
& = (1/\epsilon^2) \| \rho_{K_1} - \rho_{K_2} \|_{\infty}.
\end{align*}
\end{proof}

The basic idea behind the proof of Theorem \ref{thm:optstarbodyreg} is via contradiction.  Suppose that $\hat{K}$ is not optimal, and that there is a different star body $\tilde{K}$ that attains a smaller objective.  We pass over to the finite dimensional problem to arrive at a contradiction.  Concretely, we use $\tilde{K}$ to construct a $\tilde{K}(\mathcal{U}_t)$ that is piecewise constant on $\mathcal{U}_t$, and for a partition $\mathcal{U}_t$ that is suitably fine.

%

\begin{proof}[Proof of Theorem \ref{thm:optstarbodyreg}]

Let $\{ \mathcal{U}_t \}_{t=1}^{\infty}$ be a refining partition.  Let $\hat{K} (\mathcal{U}_t)$ be the optimal solution to \eqref{eq:optstar_piecewiseconstant} corresponding to each $\mathcal{U}_t$.  

In the first part, we prove that $\hat{K}$ is optimal.  Suppose this is not the case, and that there exists a different star body $\tilde{K}$ with unit volume and a strictly smaller objective $\E_P[\|\bx\|_K]$.  One can perturb $\tilde{K}$ so that it contains a very small kernel, while still attaining an objective that is still strictly smaller than that of $\hat{K}$, while having unit volume.  We assume this is the case.

Now define
\begin{equation*}
\epsilon := \E_P[\|\bx\|_{\hat{K}}] - \E_P[\|\bx\|_{\tilde{K}}] > 0.
\end{equation*}

Then, because $\{ \mathcal{U}_t \}_{t=1}^{\infty}$ is a refining partition, we have $\| \rho_{\hat{K} (\mathcal{U}_t)} - \rho_{\hat{K}} \|_{\infty} \rightarrow 0$ as $t \rightarrow \infty$.  In particular, there exists $t_0$ such that for all $t \geq t_0$, one has
\begin{equation*}
|\E_P[\|\bx\|_{\hat{K}}] - \E_P[\|\bx\|_{\hat{K}(\mathcal{U}_t)}] | = |\E_P[\|\bx\|_{\hat{K}}] -  \sum_{u \in \mathcal{U}_t} a_U \hat{t}_U | \leq \epsilon / 3.
\end{equation*}  

Next, by using Lemma \ref{thm:approxstarbody_piecewise} with the choice of $K$ being $\tilde{K}$, we construct the sequence $\{ \tilde{K} (\mathcal{U}_t) \}$. In particular, since $\tilde{K}$ is a star body, there exists a $\delta >0$ such that $\delta B^d \subseteq \tilde{K}$. Choosing a partition $\mathcal{U}_t$ such that $\frac{\delta}{2} B^d \subseteq \tilde{K}(\mathcal{U}_t)$, one has
\begin{equation*}
\begin{aligned}
\big| \mathbb{E}[ \| \bx \|_{\tilde{K}} ] - \mathbb{E}[ \| \bx \|_{\tilde{K}(\mathcal{U}_t)} ] \big| \,=~ & \big| \int_{\mathbb{S}^{d-1}} \int_{r=0}^{\infty} r^d ( \|\bv\|_{\tilde{K}} - \|\bv\|_{\tilde{K}(\mathcal{U}_t)} ) p(r\bv) \mathrm{d}\bv \big| \\
\,\overset{(a)}{\leq}~ & \| \|\cdot\|_{\tilde{K}} - \|\cdot\|_{\tilde{K}(\mathcal{U}_t)} \|_{\infty} \int_{\mathbb{S}^{d-1}} \int_{r=0}^{\infty} r^d p(r\bv) \mathrm{d}\bv \\
\,\overset{(b)}{\leq}~ & c_{\delta} \| \rho_{\tilde{K}} - \rho_{\tilde{K}(\mathcal{U}_t)} \|_{\infty} \int_{\mathbb{S}^{d-1}} \int_{r=0}^{\infty} r^d p(r\bv) \mathrm{d}\bv.
\end{aligned}
\end{equation*}
where $c_{\delta} := 4/\delta^2$. Here, we get (a) because $\|\bv\|_{\tilde{K}} - \|\bv\|_{\tilde{K}(\mathcal{U}_t)} \leq \| \|\cdot\|_{\tilde{K}} - \|\cdot\|_{\tilde{K}(\mathcal{U}_t)} \|_{\infty}$ by definition, while (b) follows from an application of Lemma \ref{thm:gaugeradial}. In particular, $\| \rho_{\tilde{K}} - \rho_{\tilde{K}(\mathcal{U}_t)} \|_{\infty}$ can be controlled by the choice of the partition $\mathcal{U}_t$. We choose it to be sufficiently fine so that one has $| \mathbb{E}[ \| \bx \|_{\tilde{K}} ] - \mathbb{E}[ \| \bx \|_{\tilde{K}(\mathcal{U}_t)} ] | \leq \epsilon / 3$.

Notice that $\tilde{K}(\mathcal{U}_t)$ has volume one, is piecewise constant over $\mathcal{U}_t$, and has an objective value that improves on the objective of $\hat{K} (\mathcal{U}_t)$ by at least $\epsilon/3$.  This contradicts the optimality of $\hat{K} (\mathcal{U}_t)$ in \eqref{eq:optformulation_piecewiseconstant} with $\mathcal{U} = \mathcal{U}_t$.  Therefore it must be that $\hat{K}$ is an optimal solution.
\end{proof}

\subsection{Critic-Based Regularizers and Robust Extensions} \label{sec:critic}

In this section, we discuss a variant of the formulation \eqref{eq:opt_star_formulation} in which we seek regularizers that are optimal in a critic-based, adversarial framework as well as its distributionally robust extensions \cite{leong2024star}.  In particular, prior work in adversarial regularization \cite{lunz2018adversarial, Mukherjeeetal2021, Shumaylovetal2024, zhang2025learning} have learned regularizers by solving an optimization problem of the form \begin{align*}
    \min_{\mathcal{R}} \mathbb{E}_P[\mathcal{R}(\bx)] - \mathbb{E}_Q[\mathcal{R}(\bx)] + \mathbb{E}[(\|\nabla \mathcal{R}(\bx)\| -1)_+]
\end{align*} where $(t)_+:=\max\{t,0\}$, $P$ and $Q$ are distributions that represent clean and noisy data, respectively and the penalty term encourages the regularizer $\mathcal{R}$ to be Lipschitz. The intuition behind this formulation is that a good regularizer should assign low (high) values to likely (unlikely) data. For our variant, we consider following optimization instance
\begin{equation} \label{eq:opt_critic_formulation}
\underset{K }{\mathrm{argmin}} ~ \E_P[\|\bx\|_K] - \E_Q[\|\bx\|_K] \quad \mathrm{s.t.} \quad \mathrm{vol}(K) = 1 , \epsilon B^d \subseteq K. 
\end{equation}

How does this differ from \eqref{eq:opt_star_formulation}?  First, the objective has an additional term $- \E_Q[\|\bx\|_K]$.  We can combine this objective with the original into a single expectation, with the modification being that the corresponding measure is necessarily {\em signed}:
\begin{equation} \label{eq:opt_critic_signed}
\underset{K}{\mathrm{argmin}} ~ \E_S[\|\bx\|_K] \quad \mathrm{s.t.} \quad \mathrm{vol}(K) = 1 , \epsilon B^d \subseteq K.  
\end{equation}
Here, $S = P-Q$.  Second, we impose the constraint that $K$ contains the scaled unit-ball. Note that for star body gauges, being $1/\epsilon$-Lipschitz is equivalent to $\epsilon B^d \subseteq \mathrm{ker}(K)$. We consider a relaxed setting with $\epsilon B^d \subseteq K$. Another reason why this is necessary can be seen by considering the following discretized problem:
\begin{equation*}
\underset{t_U > 0}{\arg \min} ~~ \sum \sigma_U t_U \qquad \mathrm{s.t.} \qquad \sum w_U (1/t_U)^d \leq 1.
\end{equation*}
As before, we define $\sigma_U = \int_{\bv \in U} \int_{r=0}^{\infty} r^d p(r\bv) \mathrm{d}\bv - \int_{\bv \in U} \int_{r=0}^{\infty} r^d q(r\bv) \mathrm{d}\bv$.  However, the key difference in this current set-up from the previous is that the $\sigma_U$'s are derived from {\em signed} measures; in particular, they could be {\em negative} in certain sectors $u$.  Suppose indeed that $\sigma_U < 0$ for some $U$.  In principle, one can take $t_U \rightarrow +\infty$ without affecting the volume constraint since $(1/t_U)^d \rightarrow 0$.  We would then have that the objective $\rightarrow -\infty$; that is, it is unbounded below.  This has a couple of consequences:  First, this means that the gauge function evaluation of $K$ is unbounded in the sector $U$, which may be somewhat undesirable from a modelling perspective.  Second, and perhaps more seriously, is that because the objective $\rightarrow -\infty$, it is difficult to reason if the optimization formulation is actually doing anything sensible.

To circumvent these problems, we impose additional constraints on the problem so that we avoid having the gauge evaluate to $+\infty$ in certain directions.  In what follows, we specifically impose that the gauge function $t_U$ takes a maximum value of $1/\epsilon$.  Stated differently, the radial distance of $K$, in the sector $u$, is at least $\epsilon$.  
By imposing this constraint, we arrive at the optimization instance
\begin{equation*}
\underset{t_U}{\arg \min} ~~ \sum \sigma_U t_U \qquad \mathrm{s.t.} \qquad \sum (1/t_U)^d \leq 1, 0 < t_U \leq 1/\epsilon.
\end{equation*}

\begin{proposition} \label{thm:optcritic_piecewiseconstant}
Let $P$ and $Q$ be distribution on $\R^d$ with densities $p$ and $q$.  Consider 
\begin{equation} \label{eq:optcritic_piecewiseconstant}
\underset{K}{\mathrm{argmin}}  ~  \E_{P-Q}[\|\bx\|_K] \quad \mathrm{s.t.} \quad \mathrm{vol}(K) = 1, \epsilon B^d\subseteq K, K \text{ is piecewise constant over } U \in \mathcal{U}.
\end{equation}
Then the optimal solution is the star set $\hat{K}$ whose radial function over each $u \in \mathcal{U}$ is given by $\rho_U$ where
\begin{equation*}
\rho_U = 
\begin{cases}
\begin{array}{ccc}
\epsilon \quad & \text{ if } & \quad \sigma_U < \epsilon^{d+1} d w_U \lambda \\
\epsilon (\sigma_U / d w_U \lambda )^{1/(d+1)} \quad & \text{ if } & \quad \sigma \geq \epsilon^{d+1} d w_U \lambda
\end{array}
\end{cases}.
\end{equation*}
where $\lambda$ is a scaling parameter such that $\sum r_U^d = 1$.
\end{proposition}

\begin{proof}[Proof of Proposition \ref{thm:optcritic_piecewiseconstant}]
By a similar reasoning, the optimization instance that captures the above problem is
\begin{equation*}
\underset{t_U}{\arg \min} ~~ \sum \sigma_U t_U \qquad \mathrm{s.t.} \qquad \sum w_U (1/t_U)^d \leq 1, 0 < t_U \leq 1/\epsilon.
\end{equation*}

For now, we ignore the constraint $t_U > 0$.  The Lagrangian is
\begin{equation*}
\mathcal{L} := \sum \sigma_U t_U + \sum \mu_U (t_U - 1/\epsilon) + \lambda \big(\sum w_U / t_U^d - 1 \big).
\end{equation*}
The derivative of $\mathcal{L}$ with respect to $t_U$ is
\begin{equation*}
\frac{ d \mathcal{L} }{d t_U } = \sigma_U + \mu_U - d w_U \lambda / t_U^{d+1}.
\end{equation*}
At optimality, one has 
\begin{equation*}
\sigma_U + \mu_U = d w_U \lambda / t_U^{d+1}.
\end{equation*}
By considering cases, one sees that
\begin{equation*}
t_U = 
\begin{cases}
(1/\epsilon) \quad & \text{ if } \quad \sigma_U < \epsilon^{d+1} d w_U \lambda \\
(1/\epsilon) \times (d w_U \lambda / \sigma_U)^{1/(d+1)} \quad & \text{ if } \quad \sigma_U \geq \epsilon^{d+1} d w_U \lambda
\end{cases}.
\end{equation*}
Here, $\lambda$ is a scaling parameter such that $\sum (1/t_U)^d = 1$.

We explain these choices of $t_U$.  In the case where $\sigma_U < \epsilon^{d+1} d w_U \lambda$, we set $t_U = 1/\epsilon$, and $\mu_U = d w_U \lambda / t_U^{d+1} - \sigma_U = d w_U \lambda \epsilon^{d+1} - \sigma_U$.    Suppose $\sigma_U \geq \epsilon^{d+1} d w_U \lambda $.  Then set $\mu_U = 0$, and $\sigma_U = d w_U \lambda / t_U^{d+1}$; that is, we set $t_U = (d w_U \lambda / \sigma_U)^{1/(d+1)}$.  These choices satisfy the first order optimality conditions, feasibility conditions, and complementary slackness.  Since the optimization instance is convex, the solution to the KKT system are indeed optimal as well.  Finally, notice that in all of these cases we always have $t_U > 0$.  As such the constraint $t_U$ is automatically satisfied and need not be enforced.
\end{proof}

Using similar intuition to the derivation of Theorem \ref{thm:optstarbodyreg}, we arrive at the following solution in the continuous case.

\begin{theorem} \label{thm:optcriticbodyreg}
Let $P$ and $Q$ be a distribution on $\R^d$ with density $p$ and $q$ respectively. Suppose $\rho_P$, $\rho_Q$ are continuous.  For each $\bv \in S^{d-1}$, define
\begin{equation*}
\sigma(\bv) = \int_0^{\infty} r^d p(r\bv) \mathrm{d} r - \int_0^{\infty} r^d q(r\bv) \mathrm{d} r.
\end{equation*}
Let $\hat{K}$ be the star body whose radial function is
\begin{equation*}
\rho(u) = 
\begin{cases}
\begin{array}{ccc}
\epsilon & \text{ if } & \sigma(U) < \epsilon^{d+1} c \\
\epsilon (\sigma(U) /c)^{1/(d+1)} & \text{ if } & \sigma(U) \geq \epsilon^{d+1} c
\end{array}
\end{cases},
\end{equation*}
where $c$ is the unique scaling parameter chosen so that $\mathrm{vol}(\hat{K}) = 1$.  Then $\hat{K}$ is the solution to the minimization problem \eqref{eq:opt_critic_formulation}.
\end{theorem}

The proof of this result follows analogously to the proof of Theorem \ref{thm:optstarbodyreg}.  As such, we omit the proof of Theorem \ref{thm:optcriticbodyreg}.

\subsubsection{Distributionally robust critic-based regularizers}

The distributionally robust counterpart to \eqref{eq:opt_critic_formulation} is the following 
\begin{equation} \label{eq:dro_critic}
\underset{K }{\mathrm{argmin}} ~ \Big[ \underset{ d (\tilde{P},P) \leq \epsilon_P, d (Q,\tilde{Q}) \leq \epsilon_Q}{\max} ~ \Big[ \E_{\tilde{P}}[\|\bx\|_K] -  \E_{\tilde{Q}}[\|\bx\|_K] \Big]\Big] \quad \mathrm{s.t.} \quad \mathrm{vol}(K) = 1 , \epsilon B^d \subseteq K. 
\end{equation}
The maximum is taken with respect to all pairs of measures $\tilde{P}$ and $\tilde{Q}$ close to $P$ and $Q$.  It reflects the worst case instances of $\tilde{P}$ and $\tilde{Q}$, given reference distributions $P$ and $Q$.

To derive an expression analogous to \eqref{thm:dro_formulation}, we first state the dual expression of the inner problem to \eqref{eq:dro_critic}.  Concretely, consider the following LP:
\begin{equation*}
\begin{aligned}
\underset{\boldsymbol \beta_p, \boldsymbol \beta_q, \pi_p, \pi_q}{\max} ~ \langle \boldsymbol \beta_p - \boldsymbol \beta_q, \boldsymbol t \rangle \quad \mathrm{s.t.} & \quad \langle C, \pi_p \rangle \leq \epsilon_p, \pi_p \mathbf{1} = \boldsymbol \alpha_p, \pi_p^T \mathbf{1} = \boldsymbol \beta_p, \pi_p \geq 0 \\
& \quad \langle C, \pi_q \rangle \leq \epsilon_q, \pi_q \mathbf{1} = \boldsymbol \alpha_q, \pi_q^T \mathbf{1} = \boldsymbol \beta_q, \pi_q \geq 0  
\end{aligned}
\end{equation*}
The dual LP to the above is:
\begin{equation*}
\begin{aligned}
\underset{\boldsymbol \lambda, s}{\min} \qquad & s_P \epsilon_P + s_Q \epsilon_Q + \langle \boldsymbol \alpha_P, \boldsymbol \lambda_P \rangle + \langle \boldsymbol \alpha_Q, \boldsymbol \lambda_Q \rangle \\
\mathrm{s.t.} \qquad &  s_P C + \boldsymbol \lambda_P \mathbf{1}^T \geq \mathbf{1} \mathbf{t}^T, s_P \geq 0 \\
& s_Q C + \boldsymbol \lambda_Q \mathbf{1}^T \geq - \mathbf{1} \mathbf{t}^T, s_Q \geq 0
\end{aligned}
\end{equation*}

With this, we are able to state an equivalent formulation of \eqref{eq:dro_critic}, purely as a minimization instance:

\begin{proposition}
Under the setting of Theorem \ref{thm:dro_formulation}, the following inner problem for a fixed star body $K$
\begin{equation*}
\underset{ d (\tilde{P},P) \leq \epsilon_P, d (Q,\tilde{Q}) \leq \epsilon_Q}{\max} ~ \Big[ \E_{\tilde{P}}[\|\bx\|_K] -  \E_{\tilde{Q}}[\|\bx\|_K] \Big]
\end{equation*}
is equivalent to
\begin{equation} \label{eq:dro_cts_critic}
\begin{aligned}
\underset{s_P, s_Q, \lambda_P, \lambda_Q \in L1 (d\mu_{\alpha})}{\mathrm{argmin}} ~~ & s_P \epsilon_P + s_Q \epsilon_Q + \int \lambda_P (\bx) \mathrm{d}P (\bx) + \int \lambda_Q (\bx) \mathrm{d}Q(\bx) \\ 
\mathrm{s.t.} ~~ \qquad ~~ & s_P C(\bx,\by) + \lambda_P (\bx) \geq \| \by \|_{K}, s_Q C(\bx,\by) + \lambda_Q (\bx) \geq - \| \by \|_{K} \\
& s_P \geq 0 , s_Q \geq 0.
\end{aligned}
\end{equation}
\end{proposition}
\section{Conclusion} \label{sec:conclusions}

We developed a framework for distributionally robust optimal regularization, providing a principled approach to design regularization functionals that remain stable under distributional uncertainty. Our main contributions are as follows. (i) We present a convex-duality reformulation of the DRO problem \eqref{eq:dro} that renders the robust optimal regularization problem computationally tractable. Then, (ii) we present structural results and empirics that reveal how distributional robustness affects the geometry of the regularizer. Finally,  (iii) we study how to incorporate convex geometric constraints into the regularizer, and (iv) provide elementary proof techniques for establishing optimality of such regularizers.

There are several promising directions for future work:

\begin{itemize}
\item \textbf{Precise forms of the optimal regularizers.} While our work focused on numerical schemes to compute the optimal regularizer and analyzed how both the robustness parameter and cost function influence its geometry, it would be interesting to be able to describe, even in specific cases, what the exact form of the optimal regularizer is. We make progress towards this in the case of the Wasserstein-1 distance in Proposition \ref{prop:convex-lipschitz-prop}, but a more general understanding for other distances would be of interest. This would be particularly interesting in the case when we enforce convexity as a geometric constraint.
\item \textbf{Theoretical robustness of convex regularizers.} Our stylized experiments suggest our proposed notion of the optimal convex regularizer as in \eqref{eq:optimalconvexregularizer} appear to be robust to changes in the underlying distribution.  It would be interesting to investigate this observation formally.  More importantly, any result that supports our observation has important implications in practical applications, as it provides a compelling reason to learn or develop convexity-based models in data analytical and machine learning problems as opposed to non-convex ones, as convexity-based models appear to be naturally robust to perturbations in the underlying data distribution; in contrast, additional interventions to promote generalization are necessary when learning non-convex models from data.

\item \textbf{Beyond Wasserstein–based ambiguity sets.} Extending the analysis to other divergence measures and transport costs could broaden the scope of applications. Moreover, it would shed light on how different types of divergences lead to different structure in the induced regularizer.

\item \textbf{Algorithmic aspects.} Developing scalable solvers for the distributionally robust program and the convex program formulations in higher dimensions is an important step for practical deployment. This would also be imperative for future work in using such regularizers in inverse problems arising in scientific contexts where robustness is important, such as medical imaging.

\end{itemize}

\bibliography{opt-reg-ii}


\end{document}